\documentclass[11pt]{amsart}

\usepackage{
  amsmath,
  amsfonts,
  amssymb,
  amsthm,
  amscd,
  graphicx,
  comment,      
  etoolbox,     
  mathtools,    
  enumitem,
  mathdots,     
  stmaryrd,     
  txfonts,      
  booktabs,     
  stackrel,     
}
\usepackage[usenames,dvipsnames]{xcolor}
\usepackage[all]{xy}

\usepackage[scaled=0.88]{beraserif} 
\usepackage[scaled=0.85]{berasans}  
\usepackage[scaled=0.84]{beramono}  
\usepackage[T1]{fontenc}
\usepackage{textcomp}               
\usepackage[sc]{mathpazo}           
\usepackage[T1,small,euler-digits,euler-hat-accent]{eulervm}  
\usepackage{mathrsfs}               

\usepackage[colorlinks=true, pdfstartview=FitV, linkcolor=blue, citecolor=blue, urlcolor=blue, breaklinks=true]{hyperref}

\usepackage[capitalize]{cleveref}   

\usepackage{tikz}
\usepackage{tikz-cd}

\newcommand\redcircle[1]{\filldraw[fill=white, draw=red] #1 circle (2pt)}
\newcommand\bluedot[1]{\filldraw[blue] #1 circle (2pt)}

\newcommand\dotlabel[1]{$\scriptstyle{#1}$}

\newcommand\regionlabel[1]{$\scriptstyle{#1}$}

\newcommand\tokenlabel[1]{$\scriptstyle{#1}$}

\tikzset{anchorbase/.style={>=To,baseline={([yshift=-0.5ex]current bounding box.center)}}}
\tikzset{wipe/.style={white,line width=4pt}}

\tikzset{->-/.style={decoration={
  markings,
  mark=at position #1 with {\arrow{>}}},postaction={decorate}}}
\tikzset{-<-/.style={decoration={
  markings,
  mark=at position #1 with {\arrow{<}}},postaction={decorate}}}

\crefname{defin}{Definition}{Definitions}
\crefname{eg}{Example}{Examples}
\crefname{lem}{Lemma}{Lemmas}
\crefname{theo}{Theorem}{Theorems}
\crefname{equation}{}{}
\crefname{enumi}{}{}

\newcommand\C{\mathbb{C}}

\newcommand\N{\mathbb{N}}
\newcommand\kk{\Bbbk}

\newcommand\AH{\mathcal{AH}}
\newcommand\AW{\mathcal{AW}}

\newcommand\cB{\mathcal{B}}
\newcommand\cH{\mathcal{H}}

\newcommand\cS{\mathcal{S}}
\newcommand\cW{\mathcal{W}}

\newcommand\ba{\mathbold{a}}
\newcommand\baone{\mathbold{a_1}}
\newcommand\batwo{\mathbold{a_2}}

\newcommand\one{\mathbb{1}}

\newcommand\dg{\mathrm{deg}}

\newcommand\tr{\mathrm{tr}}

\newcommand\rC{\mathcal{C}}
\newcommand\rD{\mathcal{D}}
\newcommand\rL{\mathcal{L}}
\newcommand\rI{\mathcal{I}}
\newcommand\rJ{\mathcal{J}}
\newcommand\rM{\mathcal{M}}
\newcommand\rN{\mathcal{N}}

\newcommand\rF{\mathcal{F}}

\newcommand\tM{\mathscr{M}}

\newcommand\tL{\mathscr{L}}

\newcommand\Quiv{\mathbf{Quiv}}
\newcommand\Cat{\mathbf{Cat}}
\newcommand\LCat{\mathbf{LinCat}}
\newcommand\MQuiv{\mathbf{MonQuiv}}
\newcommand\LMCat{\mathbf{LinMonCat}}

\def\chk#1{#1^{\smash{\scalebox{.8}[1.4]{\rotatebox{90}{\textnormal{\guilsinglleft}}}}}} 

\DeclareMathOperator{\End}{End}

\DeclareMathOperator{\Hom}{Hom}

\DeclareMathOperator{\id}{id}

\newtheorem{theo}{Theorem}[section]
\newtheorem{prop}[theo]{Proposition}
\newtheorem{lem}[theo]{Lemma}
\newtheorem*{lem*}{Lemma}

\theoremstyle{definition}
\newtheorem{defin}[theo]{Definition}
\newtheorem{rem}[theo]{Remark}

\numberwithin{equation}{section}
\allowdisplaybreaks

\setenumerate[1]{label=(\alph*)}  
\setenumerate[2]{label=(\roman*)}

\newtoggle{comments}
\newtoggle{details}
\newtoggle{detailsnote}

\toggletrue{comments}   
\toggletrue{details}   

\toggletrue{detailsnote}   

\iftoggle{comments}{%
  \newcommand{\comments}[1]{
    \ \\
    {\color{red}
      \textbf{AS:} #1
    }
    \\
    }
}{%
  \newcommand{\comments}[1]{}
}

\iftoggle{details}{%
  \newcommand{\details}[1]{
      \ \\
      {\color{OliveGreen}
        \textbf{Details:} #1
      }
      \\
  }
}{%
  \newcommand{\details}[1]{}
}

\begin{document}
%

\title{Presentations of linear monoidal categories and their endomorphism algebras}
\author{Bingyan Liu}
\address{
  Department of Mathematics and Statistics \\
  University of Ottawa
}
\address{
  School of Mathematical Science \\
  Tongji University
}
\email{bingyanliu21@gmail.com}

\begin{abstract}
  We give the definition of presentations of linear monoidal categories. Our main result is that given a presentation of a linear monoidal category, we can produce a presentation of the same category as a linear category. We apply this result to endomorphism algebras of certain important linear monoidal categories.
\end{abstract}


\maketitle
\thispagestyle{empty}

\tableofcontents

\section{Introduction}
Categorification is an exciting and relatively new area of research. It is the process of replacing set-theoretical objects by their categorical analogues. Strict linear monoidal categories have been playing an important role in the study of categorification, since they are categorification of classic algebraic objects---rings, where the tensor product is the categorification of the product in the ring.

In practice, (strict) linear monoidal categories are often defined by some generating objects and generating morphisms where objects and morphisms are formal tensor products of generators, with certain relations imposed. This leads to the concept of presentation of linear monoidal categories. Although common, the definition of presentation of linear monoidal categories and certain prerequisite concepts are not well represented in the literature. To make these notations precise is one of the goals of the current article.
We first review the notion of linear categories and define the notion of presentation of linear categories in \cref{sec:LinCat}. We denote such a presentation by $\left<X,E,R\right>$ where $X$ is a set of generating objects, $E$ is a set of generating morphisms, and $R$ is a set of relations on morphisms.
Then we build the monoidal analogue of these notions in \cref{sec: LinMonCat}. Concretely speaking, we define free linear monoidal categories, tensor ideals and presentations of linear monoidal categories, denoted by $\left<X,E,R\right>_\otimes$ if the category is generated by objects in $X$ and morphisms in $E$ with relations $R$ imposed. In particular, we show the existence of the free linear monoidal categories by concrete constructions in \cref{sec:LinMonCatExist}.

In a linear monoidal category $\rC$, the endomorphism sets $\End_\rC(a)$ form algebras, where the product is the composition of morphisms. The tensor product induces an extra structure on the endomorphism algebras. Thus to study the structure of the endomorphism algebras, it proves useful to forget the tensor product and view the category as only a linear category. Therefore, for a monoidal category given by presentation $\left<X,E,R\right>_\otimes$, we are motivated to give a presentation $\left<X^\prime,E^\prime,R^\prime\right>$ of it as a linear category and study the structure of its endomorphism algebras according to the presentation   $\left<X^\prime,E^\prime,R^\prime\right>.$

In \cref{sec:MainTheorem}, we introduce our main result \cref{MainTheorem}, which is stated without proof in \cite[\S  2.6]{Bru14}. Namely, given a presentation  of a linear monoidal category $\rC$ with presentation $\left<X,E,R\right>_\otimes$, we produce a presentation $\left<M(X),\bar{E},\bar{R}\right>$ of $\rC$ as a linear category. Roughly speaking, $M(X)$ is the free monoid over $X$, the morphisms $\bar{E}$ are morphisms of the form $1_v \otimes e \otimes 1_w$, and the relations are those obtained from $R$ by replacing each $r$ with $1_v \otimes r \otimes 1_w$ together with the interchange law for morphisms in $\bar{E}$.
For the structure of the endomorphism algebras, we consider a special case where all generating morphisms are endomorphisms. Given the presentation of the linear monoidal category $\left<X,E,R\right>_\otimes$, we describe a presentation of the endomorphism algebras, with the generators and relations of that algebra in those of the presentation $\left<M(X), \bar{E}, \bar{R}\right>$. This result is stated in  \cref{thm:presentationofalgebra}.

As important categories in the study of categorification, we introduce several examples of linear monoidal categories and visualize these categories with the help of string diagrams. As applications of our main results, we identify their endomorphism algebras with certain well-studied associative algebras. These identifications are stated in \cite[\S 3]{Savage18} without proof.
\color{black}

Fix a commutative ground ring $\kk$. In this article, all linear categories are $\kk$-linear categories and monoidal categories are strict monoidal categories.

\subsection*{Acknowledgements.} I would like to express my sincere gratitude to Professor Alistair Savage for his patience and guidance throughout the whole project. I would like to thank the Mitacs Organization, the China Scholarship Council and the University of Ottawa for providing me this great opportunity.

\section{Generators, relations and presentations of linear categories}

In this section, the main goal is to give a precise definition of presentations of ($\kk$-)linear categories. We will first review the definition of quivers, (small) linear categories and the free linear category over a quiver. Afterwards, we recall the definition of ideals of linear categories, the quotient category by an ideal and presentations of linear categories.

\subsection{Quivers and linear categories}\label{sec:LinCat}
A \emph{quiver} $Q=(V,E,\partial_0,\partial_1)$ consists of a set of vertices $V$ and a set of edges $E$, together with two functions $\partial_0, \partial_1 \colon E \to V$, which are called the start point function and end point function respectively. We write $e \colon a \to b$ if $\partial_0 e= a $ and $\partial_1 e=b$ and often write $Q=(V,E)$ for conciseness leaving the maps $\delta_0$ and $\delta_1$ implied. A \emph{quiver morphism} $\phi$ from a quiver $Q=(V,E)$ to a quiver $Q^\prime=(V^\prime,E^\prime)$ is a pair of functions $\phi=(\phi_0,\phi_1)$ where $\phi_0$ is a function from $V$ to $V^\prime$ and $\phi_1$ is a function from $E$ to $E^\prime$ such that the following two diagrams commute:
  \begin{equation}\label{QuivDiag}
    \begin{tikzcd}
    E\arrow{r}{\phi_1} \arrow{d}{\partial_0} & E^\prime\arrow{d}{\partial_0}\\
    V\arrow{r}{\phi_0}&V^\prime
    \end{tikzcd},
    \qquad
    \begin{tikzcd}
    E\arrow{r}{\phi_1}\arrow{d}{\partial_1}&E^\prime\arrow{d}{\partial_1}\\
    V\arrow{r}{\phi_0}&V^\prime
    \end{tikzcd}.
  \end{equation}

Quivers and quiver morphisms constitute the category of quivers $\Quiv$, where quivers serve as objects, quiver morphisms serve as morphisms and the composition of morphisms is the composition of functions. The composition of morphisms is clearly associative. The pair of identity maps $id_Q=(id_V,id_E)$ serves as the identity morphism over a object $Q=(V,E)$.

A (small) category is called a \emph{linear category} if its hom-sets are $\kk$-modules and the composition of morphisms is bilinear, namely
\begin{itemize}
	\item $(k_1\ f_1+k_2\ f_2) \circ g = k_1 (f_1 \circ g) + k_2 (f_2 \circ g),$
   \item $f \circ (k_1\ g_1+k_2\ g_2) = k_1 (f \circ g_1) + k_2 (f \circ g_2),$
\end{itemize}
for $k_1,k_2 \in \kk$ and $f,f_1,f_2,g,g_1,g_2$ morphisms such that the above operations are defined.

A \emph{linear functor} $F$ between linear categories is a functor satisfying \[F ( k_1\  f_1 + k_2\  f_2) = k_1 F(f_1)+k_2 F(f_2).\]
The category of  linear categories $\LCat$ has linear categories as objects and linear functors as morphisms.

We can define the \emph{forgetful functor} $U$ from the category of (small) linear categories $\LCat$ to the category of quivers $\Quiv$. It is based on the fact that any (small) linear category $\rC$ is manifestly a quiver $U\rC$ where the set of objects $Ob(\rC)$ forms the set of vertices and the edges consist of all morphisms $\Hom(\rC)$.

\begin{defin}\label{def:freelincat} The \emph{free linear category} over a quiver $Q$ is a linear category $\tL Q$ together with a quiver morphism $\iota$ from $Q$ to $U\tL Q$ that satisfies the following universal property: for any linear category $\rC$ and quiver morphism $\phi$ from $Q$ to $U\rC$, there exists a unique linear functor $\phi^\prime \colon \tL Q \to \rC$, such that $\phi$ factors through $\iota$ as $U\phi^\prime \circ \iota = \phi$, as the diagram below indicates:
    \begin{equation}
   		\begin{tikzcd}
        	& 			&&Q\arrow{rd}{\phi}\arrow{d}{\iota}&\\
            \tL Q\arrow{r}{\exists! \phi^\prime}&\rC,			&&U \tL Q\arrow[swap,dashrightarrow]{r}{U \phi^\prime}&U\rC.
        \end{tikzcd}
    \end{equation}
\end{defin}
The free linear category $\tL Q$ over a quiver $Q$ is unique up to isomorphism. This follows from the fact that $(\tL Q,\iota)$ is an initial object in the comma category $(Q \downarrow U)$, and the initial object in a category is unique up to isomorphism. One can find a detailed discussion in \cite[II.8]{MacLaneCat}.

The existence of $\tL Q$ also holds. Before we show this, we need the concept of free category. The \emph{free category} $\rF$ over a quiver $Q=(V,E)$ is the category whose objects are $V$ and morphisms are formal finite compositions of composable edges, where two edges $e_1, e_2$ are composable if $\partial_0 e_1 = \partial_1 e_2$. We may also call the morphisms in the free category \emph{paths} on $Q$ and call an identity morphism $1_v$ a trivial path.

Let $U_0$ be the forgetful functor from $\Cat$ to $\Quiv$ and $\iota_0$ be the evident embedding quiver morphism from $Q$ to $U_0\rF$. The free category $\rF$ together with $\iota_0$ satisfies the following universal property: given any category $\rC$ and quiver morphism $\phi \colon Q \to U_0\rC$, there exists a unique functor $\psi$ from $\rF$ to $\rC$, such that $U_0\psi \circ \iota_0 = \phi$. One can find a detailed discussion of free category in \cite[II.7]{MacLaneCat}.

\begin{theo}
Let $\rF$ be the free category over a quiver $Q$ and $\rN$ be the category whose objects are the same as those of $\rF$ and whose hom-sets are the free $\kk$-modules on the hom-sets of $\rF$. Then $\rN$ is the free linear category $\tL Q$.
\end{theo}
\begin{proof}
Let $\Cat$ be the category of (small) categories. Consider the natural forgetful functors in the following diagram:
     \[\begin{tikzcd}
        & \Cat  \arrow{dr}{U_0} &\\
     	\LCat \arrow{ur}{U_1} \arrow {rr}{U}&  & \Quiv.
     \end{tikzcd}\]
It is clear that this diagram commutes. Let $\iota_0$ be the natural embedding quiver morphism from $Q$ to $U_0\rF$ and $\iota_1$ be the natural embedding functor from $\rF$ to $U_1\rN$. We define  $\iota \colon Q \to U\rN$ to be $\iota = U_0\iota_1 \circ \iota_0 $.

Now, for any linear category $\rC$ and quiver morphism $\phi$ from $Q$ to $U\rC$, by the universal property of $\rF$, there exists a unique functor $\psi \colon \rF \to U_1 \rC$, such that $ \phi = U_0 \psi \circ \iota_0 $, as the diagram below indicates. For $\psi$, since the hom-sets of $\rN$ are free $\kk$-modules, it's easy to show there exists a unique linear functor $\phi^\prime \colon \rN \to \rC$, such that $\psi =U_1\phi^\prime \circ \iota_1$, as the diagram below indicates.
\begin{equation}\label{diag:freelincatexsitprove}
    \begin{tikzcd}
    	 Q\arrow{r}{\iota_0}\arrow[swap]{rd}{\phi}  &
         U_0 \rF\arrow[dashrightarrow]{d}{U_0\psi}  & \rF\arrow[dashrightarrow]{d}{\psi}
       \\
       & UL=U_0U_1\rC &  U_1\rC
    \end{tikzcd},\quad
	\begin{tikzcd}
	    \rF\arrow{r}{\iota_1}\arrow[swap]{rd}{\psi}
    	& U_1\rN \arrow[dashrightarrow]{d}{U_1\phi^\prime}& \rN \arrow[dashrightarrow]{d}{\phi^\prime}&  \\
       & U_1\rC & \rC
    \end{tikzcd}.
\end{equation}

Thus, applying $U_0$ to the diagram of $\phi^\prime$ on the right in \cref{diag:freelincatexsitprove}, and connecting it with the other diagram, we have that the diagram below commutes.
\begin{equation}\label{diag:freelincatexistprove2}
	\begin{tikzcd}
    	Q \arrow[swap]{rd}{\phi}\arrow{r}{\iota_0}
        & U_0\rF \arrow[dashrightarrow]{d}{U_0\psi} \arrow{r}{U_0\iota_1}
        & U_0U_1\rN \arrow[dashrightarrow]{ld}{U_0U_1\phi^\prime}\\
        & U_0U_1\rC &
    \end{tikzcd}
\end{equation}
Thus we have $\phi= U_0U_1 \phi^\prime \circ U_0 \iota_1 \circ \iota_0 = U\phi^\prime \circ \iota.$

Suppose there exists another linear functor $\psi^\prime$ which satisfies $\phi = U \psi^\prime \circ \iota$. We have $ \phi = U_0 ( U_1 \psi^\prime \circ \iota_1 )\circ \iota_0$. By the uniqueness of $\psi$, we have $ U_1 \psi^\prime \circ \iota_1 = \psi$. By the uniqueness of $\phi^\prime$, we have $\psi^\prime=\phi^\prime$. Thus $\rN=\tL Q$.
\end{proof}

As constructed above, the hom-set ${\tL Q}(a,b)$ for any objects $a$, $b$ are free $\kk$-modules over the hom-sets ${\rF}(a,b)$ of $\rF$. Therefore morphisms in $\tL Q$ can always be uniquely written as linear combinations of morphisms in $\rF$, namely linear combinations of paths on $Q$.

\subsection {Ideals, relations and presentations} \label{sec:LinCatPre}
\begin{defin}
	An \emph{ideal} $\rI$ of a linear category $\rL$ is a linear subcategory $\rI$ such that
\begin{itemize}
	\item for all $f \in \rI(a,b)$, $g \in \rL(b,c)$, we have $g \circ f \in \rI(a,c)$,
   \item for all $f \in \rL(a,b)$, $g \in \rI(b,c)$, we have $g \circ f \in \rI(a,c)$.
\end{itemize}
Let $R$ be a set of morphisms of a linear category $\rL$, not necessary in a single hom-set.
The \emph{ideal generated by} $R$ is the minimal ideal that contains $R$, denoted as $\left<R\right>$,  which can also be described as the intersection of all ideals that contain $R$.
\end{defin}
To give a explicit description of the generated ideal generated by a set of morphisms, we have the following proposition.

\begin{prop}\label{idealstruc}
Let $R$ be a set of morphisms of a linear category $\rC$. Then the ideal $\left<R\right>$ consists of all morphisms of the form
	\begin{equation} \label{equ:idealstruc}
    	\sum_{i=1}^n k_i\  (f^\prime_i\circ r_i \circ f_i),
    \end{equation}
where $n \in \N$, $k_i \in \kk$, $r_i \in R$ and $f_i$, $f^\prime_i$ are morphisms in $\rC$ such that the above operations are defined.
\end{prop}
\begin{proof}
Let $\rI$ denote the category with all morphisms of the form \cref{equ:idealstruc}. Then $\rI$ forms a linear subcategory.
Since the composition of morphisms is bilinear, we have
	\begin{align*}
    	g^\prime \circ \left(\sum_{i=1}^n k_i (f^\prime_i\circ r_i \circ f_i)\right) \circ g
       =\sum_{i=1}^n k_i (g^\prime \circ f^\prime_i) \circ r_i \circ (f_i \circ f^\prime_i) \in \Hom(\rI),
	\end{align*}
for any morphisms $g$, $g^\prime$ such that the operations above are defined. Thus $\rI$ forms an ideal and $\rI$ contains $R$. Hence $\rI$ contains $\left<R\right>$. Conversely, since $R$ is contained in $\left<R\right>$ and $\left<R\right>$ is an ideal, we have $f^\prime_i\circ r_i \circ f_i \in \left<R\right>$ for any $f$,$f^\prime$ such that these compositions are defined. Since $\left<R\right>$ is closed under linear operations, all morphisms in \cref{idealstruc} are in $\left<R\right>$, namely $\rI \subseteq \left<R\right>$. Thus $\rI=\left<R\right>$.
\end{proof}

\begin{defin}
	The \emph{quotient category} of a linear category $\rC$ by an ideal $\rI$ is the linear category $\rC/\rI$ whose objects are the same as those of $\rC$ and whose hom-sets are the quotient $\kk$-modules $\rC(a,b)/\rI(a,b)$. For $f \in \rC(a,b)$, let $[f]$ denote the coset of $f$. Then the linear operation and composition of morphisms are defined as follows:
    \begin{align*}
      [f]+[g]=[f+g], \quad k[f]=[kf], \quad [g] \circ [f]=[g \circ f],
    \end{align*}
for $k\in \kk$ and $f,g$ morphisms in $\rC$, such that the above operations are defined.
\end{defin}
The linear operations are exactly the linear operations of the quotient module. For the composition of morphisms, the definition of ideals guarantees that it is well-defined: Let $[f]=[f^\prime]$, $[g]=[g^\prime]$, then $f-f^\prime, g-g^\prime \in \Hom(\rI)$, we have
	\begin{align*}
    	g \circ f - g^\prime \circ f^\prime
        = g \circ f - g \circ f^\prime + g \circ f^\prime - g^\prime \circ f^\prime
        = g \circ (f - f^\prime) + (g - g^\prime) \circ f  \in \Hom(\rI),
    \end{align*}
therefore $[g \circ f] = [g^\prime \circ f^\prime]$.

Now we define presentations of linear categories.
\begin{defin}
A \emph{presentation} of a linear category $\rC$ is a triple $(V,E,R)$, where $R$ is a set of morpihsms of the free linear category $\tL Q$ over the quiver $Q=(V,E)$, such that $\rC$ is isomorphic to the quotient category of $\tL Q$ by the ideal generated by $R$, namely the linear category $\tL Q / \left<R\right>$.
\end{defin}

\section{Presentations of linear monoidal categories}

In this section, our main goal is to give a precise definition of presentations of (strict) linear monoidal categories. This section is developed by making a monoidal analogue of the previous section. We will first define monoidal quivers, small linear monoidal categories and the free linear monoidal category over a quiver. The existence of the free linear monoidal category is demonstrated in \cref{sec:LinMonCatExist}. Afterwards, we will define tensor ideals of linear monoidal categories, the quotient category by a tensor ideal and presentations of linear monoidal categories.

\subsection{Monoidal quivers and linear monoidal categories} \label{sec: LinMonCat}
We make a monoidal analogue of quivers, linear categories, the forgetful functor and the free linear category.

A \emph{monoidal quiver} $Q=(V,E)$ is defined to be a quiver whose vertex set $V$ is a monoid. A \emph{monoidal quiver morphism} is a quiver morphism such that the map on vertices is also a monoidal morphism. They constitute the category $\MQuiv$ of monoidal quivers.

A \emph{(strict) linear monoidal category} $\rD$ is a linear category equipped with a tensor product (bifunctor) $\otimes$ such that
\begin{itemize}
	\item $\otimes$ is associative on both objects and morphisms, and
	\item there exists $\one \in Ob(\rD)$ such that $\one \otimes a = a = a \otimes \one$ for any object $a \in Ob(\rD)$
    and $1_\one \otimes f = f = f \otimes 1_\one$ for any morphism $f  \in  \Hom(\rD)$, and
   \item the bifunctor $\otimes$ is bilinear, namely
   	 	\begin{align*}
				(k_1\ f_1+k_2\ f_2) \otimes g = k_1 (f_1 \otimes g) + k_2 (f_2 \otimes g),\\
  			    f \otimes (k_1\ g_1+k_2\ g_2) = k_1 (f \otimes g_2) + k_2 (f \otimes g_2),
		 \end{align*}
    where $k_1,k_2 \in \kk$ and $f,f_1,f_2,g,g_1,g_2$ are morphisms such that the above operations are defined.
\end{itemize}

A \emph{linear monoidal functor} $F$ between linear monoidal categories is a linear functor $F$ such that $F$ also commutes with $\otimes$, namely $F(a\otimes b) = F(a) \otimes F(b)$ for any objects $a,b$ and $F(f \otimes g) = F(f) \otimes F(g)$ for any morphisms $f,g$.
The category of small linear monoidal categories $\LMCat$ consists of linear monoidal categories as objects and linear monoidal functors as morphisms.

We can define the \emph{forgetful functor} $U$ from the category of small linear monoidal categories $\LMCat$ to the category of monoidal quivers $\MQuiv$. It is based on the fact that any (small) linear monoidal category $\rD$ is manifestly a monoidal quiver $U\rD$ where the vertice are $Ob(\rD)$ and edges consists of all morphisms $\Hom(\rD)$.

\begin{defin}\label{def:freelinmoncat} The \emph{free linear monoidal category} over a monoidal quiver $Q$ is a category $\tM Q$ together with a monoidal quiver morphism $\tau$ from $Q$ to $U\tM Q$, where $U$ is the forgetful functor defined above, that satisfies the following universal property: for any linear monoidal category $\rD$ and monoidal quiver morphism $\phi$ from $Q$ to $U\rD$, there exists a unique linear functor $\phi^\prime \colon \tM Q \to \rD$, such that $\phi$ factors through $\tau$ as $U\phi^\prime \circ \tau = \phi$, as the diagram below indicates:
    \begin{equation}
   		\begin{tikzcd}
        	& 			&&Q\arrow{rd}{\phi}\arrow{d}{\tau}&\\
            \tM Q\arrow{r}{\phi^\prime}&\rD,			&&U \tM Q\arrow[swap,dashrightarrow]{r}{U \phi^\prime}&U\rD.
        \end{tikzcd}
    \end{equation}
\end{defin}

The free linear monoidal category $\tM Q$ is unique up to isomorphism by a conventional proof. The question is does the free linear monoidal category always exist? Fortunately, the answer is \emph{yes}. We will give a detailed construction in \cref{sec:LinMonCatExist}.

\subsection{Tensor ideals, relations and presentations}

\begin{defin}
	A \emph{tensor ideal} of a linear monoidal category $\rD$ is a linear monoidal subcategory $\rI$ such that
\begin{itemize}
	\item $\rI$ is an ideal of $\rD$ when $\rI$ and $\rD$ are viewed as merely linear categories, and
	\item it satisfy the ideal property, namely
	\begin{align*}
         g \otimes f \in \rI(a \otimes c,b \otimes d),\ \text{for all } f \in \rI(a,b),\ g \in \rD(c,d)\\
         g \otimes f \in \rI(a \otimes c,b \otimes d),\ \text{for all } f \in \rD(a,b),\  g \in \rI(c,d)
	\end{align*}
\end{itemize}

Let $R$ be a set of morphisms of a linear monoidal category $\rD$, not necessary in a single hom-set.
The \emph{tensor ideal generated by} $R$ is the minimal tensor ideal that contains $R$, denoted as $\left<R\right>_\otimes$, which can also be described as the intersection of all tensor ideals that contain $R$.
\end{defin}

The following proposition proves useful.
\begin{prop} \label{prop:tenequivdef}
Let $\rD$ be a linear monoidal category. Suppose a linear monoidal subcategory $\rI$ satisfies the following conditions:
  \begin{itemize}
  \item $\rI$ is an ideal of $\rD$ when $\rI$ and $\rD$ are viewed as merely linear categories.
  \item $\rI$ satisfies the ideal property for identity morphisms; namely for any objects $c$ and morphisms $f \in \rI(a,b)$,
              we have
    \[1_c \otimes f \in \rI(c\otimes a, c \otimes b)\quad \text{ and } \quad f \otimes 1_c \in \rI( a \otimes c, b \otimes c). \]
  \end{itemize}
Then $\rI$ is a tensor ideal of $\rD$.
\end{prop}
\begin{proof}
Notice that $f\otimes g=(f\otimes 1_{\partial_1 g}) \circ (1_{\partial_0 f} \otimes g)$. Since $\rI$ satisfies the ideal property for identity morphisms, it satisfies the ideal property for any morphisms.
\end{proof}

To give an explicit description of the tensor ideal generated by a set of morphisms, we have the following proposition:

\begin{prop}\label{tenidealstruc}
Let $R$ be a set of morphisms of a linear monoidal category $\rD$. Then $\left<R\right>_\otimes$ consists of all morphisms of the form:
	\begin{equation}\label{equ:tenidealstruc}
    	\sum_{i=1}^n k_i\  f^\prime_i\circ (g^\prime_i \otimes r_i \otimes g_i) \circ f_i,
    \end{equation}
where $n \in \N$, $k_i \in \kk$, $r_i \in R$ and $f_i,f^\prime_i,g_i,g^\prime_i$ are morphisms in $\rD$ such that the above operations are defined.	
\end{prop}
\begin{proof}
Let $\rI$ denote the subcategory of all morphisms of the form in \cref{equ:tenidealstruc}. It forms an ideal by the same argument as in the proof of \cref{idealstruc}. For any two objects $a,a^\prime$ in $\rD$, since the tensor product is bilinear, we have
  \begin{align*}
		&1_a \otimes \left(\sum_{i=1}^n k_i\  f^\prime_i\circ \left(g^\prime_i \otimes r_i \otimes g_i\right) \circ f_i \right) \otimes 1_a^\prime \\
    =&  \sum_{i=1}^n k_i\  1_a\otimes  \left(f^\prime_i\circ \left(g^\prime_i \otimes r_i \otimes g_i\right) \circ f_i \right) \otimes 1_{a^\prime}\\
    =&   \sum_{i=1}^n k_i\
    \left( 1_{a} \otimes f^\prime_i \otimes 1_{a^\prime} \right)
    \circ
    \left(1_a \otimes \left(g^\prime_i \otimes r_i \otimes g_i\right)  \otimes 1_{a^\prime} \right)
    \circ
    \left(1_{a} \otimes f_i \otimes 1_{a^\prime} \right).
  \end{align*}
Since $\rI$ is a linear subcategory and
	\begin{align*}
    	1_a \otimes \left(g^\prime_i \otimes r_i \otimes g_i \right)  \otimes 1_{a^\prime} = \left( 1_a \otimes g^\prime_i\right) \otimes r_i \otimes \left( g_i  \otimes 1_{a^\prime} \right) \in \Hom(\rI),
   \end{align*}
we have that $\rI$ satisfies the ideal property for identity morphisms. By \cref{prop:tenequivdef}, $\rI$ is a tensor ideal. Since $\rI$ contains $R$, $\rI$ contains $\left<R\right>_\otimes$. On the other hand, since $r_i$ is in $\left<R\right>_\otimes$ and $\left<R\right>_\otimes$ is a tensor ideal, morphisms of the form \cref{equ:tenidealstruc} are in $R$. Thus we have $\rI=R$.
\end{proof}

Since linear monoidal categories are linear categories and any tensor ideal is an ideal by definition, combining \cref{idealstruc} with \cref{tenidealstruc}, we have the following corollary.
\begin{prop}\label{tenidealandidealstruc}
Let $R$ be a set of morphisms of a linear monoidal category $\rD$. Then viewing $\rD$ as a linear category, we have that $\left<R\right>_\otimes$, as an ideal of $\rD$, is generated by
\[
	\tilde{R}= \left\lbrace  1_a \otimes r \otimes 1_b \mid r \in R,\  a,b \in \mathrm{Ob}(\rD) \right\rbrace  ,
\]
namely $\left<R\right>_\otimes=\left<\tilde{R}\right>.$
\end{prop}
\begin{proof}
By the structure theorem of ideals and tensor ideals, namely \cref{tenidealstruc} and \cref{idealstruc} , we have $\left<R\right>_\otimes$, as an ideal, is generated by
\[
	\left\lbrace  g^\prime \otimes r \otimes g \mid r \in R,\  g,g^\prime \in \Hom(\rD) \right\rbrace  .
\]
At the same time we have
\[
    g^\prime \otimes r \otimes g = (g^\prime \otimes 1_{x^\prime} \otimes 1_{b^\prime}) \circ (1_a \otimes r \otimes 1_b) \circ ( 1_a \otimes 1_x \otimes g) \in \left<\tilde{R}\right>,
\]
where $r \colon x \to x^\prime$, $g \colon a \to a^\prime$ and $g^\prime \colon b \to b^\prime.$ Thus we have $\left<R\right>_\otimes \subseteq \left<\tilde{R}\right>$. Conversely, $1_a \otimes r \otimes 1_b \in \left<R\right>_\otimes$ implies $\left<\tilde{R}\right> \subseteq \left<R\right>_\otimes$, and so we have $\left<R\right>_\otimes = \left<\tilde{R}\right>$.
\end{proof}

\color{black}

\begin{defin}
	The \emph{quotient category} $\rD/\rI$ of a linear monoidal category $\rD$ by a tensor ideal $\rI$ is the quotient category of $\rD$ by $\rI$ viewed respectively as a linear category and an ideal, equipped with a tensor product defined by
    \begin{align*}
    	[f] \otimes [g] = [f \otimes g]
    \end{align*}
for $f,g$ morphisms in $\rD$.
\end{defin}
The definition of ideals of linear categories guarantees that the tensor product of cosets are well-defined: let $[f]=[f^\prime]$, $[g]=[g^\prime]$. Then we have $f - f^\prime \in \Hom(\rI)$, $g- g ^\prime \in \Hom(\rI)$. Thus
	\begin{align*}
		f \otimes g - f^\prime \otimes g^\prime
       = f \otimes g - f^\prime \otimes g + f^\prime \otimes g - f^\prime \otimes g^\prime
       = (f - f^\prime) \otimes g - f^\prime \otimes (g -g^\prime) \in \Hom(\rI).
	\end{align*}
Thus we have $ [f \otimes g] = [f ^\prime \otimes g^\prime]$.

For a set $X$, let $M(X)$ denote the free monoid generated by $X$. Now we define presentations of linear monoidal categories.
\begin{defin}
A \emph{presentation} of a linear monoidal category $\rD$ is a triple $(X,E,R)_\otimes$, where $R$ is a set of morphisms of the free linear monoidal category $\tM Q$ over the monoidal quiver $Q=(M(X),E)$ such that $\rD$ is isomorphic to the quotient category of $\tM Q$ by the tensor ideal generated by $R$, namely the linear monoidal category $\tM Q / \left<R\right>_\otimes$.
\end{defin}

\section{Construction of the free linear monoidal category} \label{sec:LinMonCatExist}
In this section, we prove the existence of the free (strict) linear monoidal category over a monoidal quiver by giving an explicit construction. We follow the work of \cite[\S 5]{Pie17}.

\subsection{Step one}
Let $V$ be a monoid where the product is denoted by concatenation and let $Q=(V,E)$ be a quiver on $V$. We define a new quiver $\bar{Q}=(V,\bar{E})$, where the set of edges $\bar{E}$ consists of all triples
  \[	
  		(v,e,w), \qquad v,w \in V,\  e \in E.
  \]
For an edge $e \colon x \to y$, we define
  \[	
  		\partial_0(v,e,w)=vxw,\qquad \partial_1 (v,e,w) = vyw.
  \]
Now take $\rM^\prime$ to be the free linear category over $\bar{Q}$, namely $\rM^\prime = \tL \bar{Q}$. Then all morphisms in $\rM^\prime$ are linear combinations of paths on $\bar{Q}$. 
We now define a product $\otimes$ on both objects and morphisms of $\rM^\prime$ inductively.
\begin{itemize}
    \item For two objects $a$ and $b$ of $\rM^\prime$, we define
    \[
        a \otimes b = ab, \quad 1_a \otimes 1_b = 1_{ab}.
    \]
    \item For a triple $(v,e,w)$ and a path $f=\alpha_1 \circ \dotsb \circ \alpha_n$ where $\alpha_i=(v_i,e_i,w_i),i=1,...,n$ are edges, we define
    \begin{align}
    \begin{split}
        1_a \otimes (v,e,w) = (av,e,w)\quad \text{and}\quad
        1_a \otimes f= (1_a \otimes \alpha_1) \circ \dotsb \circ (1_a \otimes \alpha_n),\\
          (v,e,w) \otimes 1_a = (v,e,w a)\quad \text{and}\quad
       f \otimes  1_a = ( \alpha_1 \otimes 1_a) \circ \dotsb \circ ( \alpha_n \otimes 1_a).\\
    \end{split}
    \end{align}
    \item For two paths $f \colon a \to b$ and $g \colon c \to d$, we define
    \begin{align} \label{productreduce}
    \begin{split}
        f \otimes g = (f \otimes 1_d) \circ (1_a \otimes g)
    \end{split}
    \end{align}
   (From now on we adopt the convention that the product $\otimes$ will be perfomed before the composition $\circ$ to avoid unnecessary parentheses.) In particular, for two edges $(v,e,w)$ and $(v^\prime,e^\prime,w^\prime)$, where $e \colon x \to y$ and $e^\prime \colon x^\prime \to y^\prime$, we have that
    \begin{align}
    \begin{split}
        (v,e,w) \otimes (v^\prime,e^\prime,w^\prime)
        =&(v,e,w) \otimes 1_{v^\prime y^\prime w^\prime} \circ 1_{vxw} \otimes (v^\prime,e^\prime,w^\prime)\\
        =&(v,e,wv^\prime y^\prime w^\prime) \circ (vxwv^\prime,e^\prime,w^\prime).
    \end{split}
    \end{align}
    \item We extend $\otimes$ to all morphisms of $\rM^\prime$ by bilinearity.
\end{itemize}
It's easy to check that $\otimes$ is associative and the identity $\one$ of the monoid $M(X)$ and the identity morphisms $1_{\one}$ on it  are the identities of objects and morphisms of the product, respectively. We may be attempted to think $(M^\prime,\otimes)$ is the free linear monoidal category. However, it is not, since the interchange law doesn't hold. In particular, the interchange law for triples fails by definition:
\begin{align}
    (v,e,wv^\prime y^\prime w^\prime) \circ (vxwv^\prime,e^\prime,w^\prime) \neq
     (vywv^\prime,e^\prime,w^\prime) \circ  (v,e,wv^\prime x^\prime w^\prime).
\end{align}

Thus we are motivated to mod out the differences of the two sides of the above expression.

\subsection{Step two} \label{sec:freelinmoncatrel}

Let $C$ be 
all morphisms of the following form
\begin{equation}\label{interchangelawoftriples}
        (v,e,wv^\prime y^\prime w^\prime) \circ (vxwv^\prime,e^\prime,w^\prime) -
     (vywv^\prime,e^\prime,w^\prime) \circ  (v,e,wv^\prime x^\prime w^\prime),
\end{equation}
where $v,w,v^\prime, w^\prime \in \mathrm{Ob}(M)$ and $e \colon x \to y, e^\prime \colon x^\prime \to y^\prime \in E$.

Let $\rM$ be the quotient category of $\rM^\prime$ by the ideal generated by $C$, namely $\rM=\rM^\prime/\left<C\right>$. Then $\otimes$ induces a product on $\rM$, still denoted by $\otimes$.

Before we show that $(\rM,\otimes)$ actually forms a linear monoidal category, we need to show that $\otimes$ is well defined on $\rM$. Namely, for a morphism $f$ in $\left<C\right>$ and an arbitrary morphism $g$ we have that $f \otimes g \in \left<C\right>$ and $g \otimes f \in \left<C\right>$.

By \cref{idealstruc}, morphisms in $\left<C\right>$ are of the form
\[
     f=\sum_{i=1}^n k_i\  p_i^\prime \circ r_i \circ p_i,
\]
where $k_i \in \kk$, $r_i \in C$ and $p_i,p_i^\prime$ are paths in $\rM^\prime$. Let $a$ and $b$ be the domain and codomain of $f$, respectively. Since for an arbitrary path $g\colon c \to d$ in $\rM^\prime$ we have
\begin{align}
\begin{split}
    f \otimes g &= f \otimes 1_d \circ 1_a \otimes g\\
                &= \sum_{i=1}^n k_i\  (p_i^\prime \circ r_i \circ p_i) \otimes 1_d \circ  1_a \otimes g\\
                &= \sum_{i=1}^n k_i\  p_i^\prime \otimes 1_d \circ r_i \otimes 1_d \circ p_i \otimes 1_d \circ 1_a \otimes g,
\end{split}
\end{align}
to show $f\otimes g$ is in $\left<C\right>$, it suffices to show that morphisms of the form $r \otimes 1_d$ are in $\left<C\right>$.
By the definition of $C$, we have $$r=\alpha \otimes 1_{y^\prime} \circ 1_{x} \otimes \alpha^\prime -
1_{y} \otimes \alpha^\prime \circ \alpha \otimes 1_{x^\prime} ,$$ for some triples $\alpha \colon x \to y$, $\alpha^\prime \colon x^\prime \to y^\prime$. Let $\beta^\prime=\alpha^\prime \otimes 1_d \colon x^\prime d \to y^\prime d$. Then we have
\begin{align}
\begin{split}
    r \otimes 1_d &=
      \alpha \otimes 1_{y^\prime} \otimes 1_d \circ 1_{x} \otimes \alpha^\prime \otimes 1_d -
1_{y} \otimes \alpha^\prime \otimes 1_d \circ \alpha \otimes 1_{x^\prime} \otimes 1_d\\
    &= \alpha \otimes 1_{y^\prime d} \circ 1_x \otimes \beta^\prime - 1_y \otimes \beta^\prime \circ \alpha \otimes 1_{x^\prime d} \in C.
\end{split}
\end{align}
Thus we have $f \otimes g \in \left<C\right>$, and similar arguements yield that $g \otimes f$ is also in $\left<C\right>$.
Thus $\otimes$ is well-defined on $\rM$. It follows directly that  $\otimes$ on $\rM$ is bilinear and associative since $\otimes$ on $\rM^\prime$ is bilinear and associative.

 Morphisms in $\rM$ are cosets of morphisms in $\rM^\prime$, denoted by $[f]$, where $f \in \Hom(\rM^\prime)$. We will omit the brackets for identity morphisms and triples to simplify our notation.

 Since we define $\rM$ by modding out morphisms of the form \cref{interchangelawoftriples}, in $\rM$ we already have the interchange law for triples:
 \begin{align}
    (v,e,wv^\prime y^\prime w^\prime) \circ (vxwv^\prime,e^\prime,w^\prime) =
     (vywv^\prime,e^\prime,w^\prime) \circ  (v,e,wv^\prime x^\prime w^\prime).
\end{align}

It quickly follows that the interchange law for arbitrary morphisms also holds, as we now explain. Since $\otimes$ is bilinear, it suffices to check the interchange law for arbitrary paths. For two paths $f=\alpha_1 \circ \dotsb \alpha_n$ and $g=\beta_1 \circ \dotsb
\circ \beta_m$ where $\alpha_i: x_i \to y_i $ and $\beta_j: z_i \to w_i$ are triples, we have that
\[
    [f] \otimes 1_{w_1} \circ 1_{x_n} \otimes [g]  = (\alpha_1 \otimes 1_{w_1}) \circ  \dotsb \circ (\alpha_n \otimes 1_{w_1}) \circ (1_{x_n} \otimes \beta_1) \circ \dotsb \circ (1_{x_n} \otimes \beta_m).
\]
By using the interchange law for triples repeatedly we can move all morphisms of the form $1_{x_i} \otimes \beta_j$ to the left. After that, the equation above reads exactly $1_{y_1} \otimes [g] \circ [f] \otimes 1_{z_m}$.
Thus the interchange law for arbitrary tensor products holds. It follows that $\otimes$ is a bifunctor since for $f_i \colon a_i \to b_i,i=1,2,$ and $g_i \colon c_i \to d_i,i=1,2$, we have
\begin{align}
\begin{split}
    ([f_1] \circ [f_2])\otimes( [g_1] \circ [g_2])
    &= ([f_1] \circ [f_2]) \otimes 1_{d_1} \circ 1_{a_2} \otimes([g_1] \circ [g_2])\\
    &= [f_1] \otimes 1_{d_1} \circ [f_2] \otimes 1_{d_1} \circ 1_{a_2} \otimes [g_1] \circ 1_{a_2} \otimes [g_2]\\
    &\stackrel{\star}{=} [f_1] \otimes 1_{d_1} \circ 1_{b_2} \otimes [g_1] \circ [f_2] \otimes 1_{z_1} \circ 1_{a_2} \otimes [g_2] \\
    &= [f_1] \otimes 1_{d_1} \circ 1_{a_1} \otimes [g_1] \circ [f_2] \otimes 1_{w_2} \circ 1_{a_2} \otimes [g_2] \\
    &= [f_1] \otimes [g_1] \circ [f_2] \otimes [g_2],
\end{split}
\end{align}
where $\star$ is due to the interchange law.

Now $(\rM,\otimes)$ does form a linear monoidal category.
\color{black}	
\subsection{Step three: universal property}

Now we claim that $(\rM,\otimes)$ satisfies the universal property in \cref{def:freelinmoncat} and therefore it is the free linear monoidal category. We denote by $\tau=(\tau_0,\tau_1)$ the natural inclusion from the base monoidal quiver $Q=(V,E)$ to $U\rM$, where $U$ is the natural forgetful functor, $\tau_0$ is identity on $V$ and $\tau_1$ maps $e$ to $(\one,e,\one)$.
\begin{theo}[Universal property]
The category $\rM$ constructed above is the free linear monoidal category over $Q$, as defined in \cref{def:freelinmoncat}.
\end{theo}
\begin{proof}
To prove that $\rM$ is the free linear monoidal category, we should show that for any linear monoidal category $\rD$ and monoidal quiver morphism $\phi$ from $Q$ to $U\rD$, where U is the forgetful functor from $\LMCat$ to $\MQuiv$, there exists a unique linear monoidal functor $\phi^\prime \colon \rM \to \rD$ such that $\phi=U\phi^\prime \circ \tau$, as the diagram below indicates:
    \begin{equation}
   		\begin{tikzcd}
        	& 			&&Q\arrow{rd}{\phi}\arrow{d}{\tau}&\\
            \rM\arrow{r}{\exists! \phi^\prime}&\rD,			&&U\rM\arrow[swap,dashrightarrow]{r}{U\phi^\prime}&U\rD.
        \end{tikzcd}
   \end{equation}

Define a linear monoidal functor $\phi^\prime$ in the following way: for $[f] = \sum_{i=1}^n k_i [f_i]$, where $f_i$ are paths and $[f_i] = \alpha^i_{1} \circ \dotsb \circ \alpha^i_{n_i}$, with $\alpha^i_{j}=(v^i_j,e^i_j,w^i_j)=1_{v^i_j} \otimes (\one,e^i_j,\one) \otimes 1_{w^i_j},$ we define $\phi^\prime([f])$ as follows:
	\begin{align*}
    \begin{split}
    	\phi^\prime([f]) &= \sum_{i=1}^n k_i\  \phi^\prime([f_i]),\\
         \phi^\prime([f_i])  &=\phi^\prime(\alpha^i_1) \circ \dotsb \circ \phi^\prime(\alpha^i_{n_i}),\\
          \phi^\prime(\alpha^i_j) &=\phi^\prime(1_{v^i_j}) \otimes \phi^\prime(\one,e^i_j,\one) \otimes \phi^\prime(1_{w^i_j}),\\
          \phi^\prime(1_{v^i_j})&=1_{\phi(v^i_j)},\quad  \phi^\prime (\one,e^i_j,\one) = \phi(e^i_j).
    \end{split}
   	\end{align*}
To show $\phi^\prime$ is well-defined, it suffices to show for any morphism $[f]=[0]$, we have $\phi^\prime([f])=0$.
By the choice of $f$ we have $f \in \left<C\right>$, where $C$ is defined in \cref{interchangelawoftriples}, and thus there exists some $g_i,g^\prime_i,r_i \in \rM$ and $k_i \in \kk$ such that $[f]=\sum_{i=1}^n k_i(g^\prime_i \circ r_i \circ g_i)$. By the definition of $\phi^\prime$, we have
	\begin{equation}\label{phiwelldefined}
	 	\phi^\prime([f])=\sum_{i=1}^n k_i \phi^\prime([g^\prime_i]) \circ \phi^\prime([r_i]) \circ \phi^\prime([g_i]).
	\end{equation}
Consider $\phi^\prime(r)$ for any $r \in C$. There exists $(v,e,w)$ and $(v^\prime,e^\prime,w^\prime)$, where $e \colon x \to y$ and $e^\prime \colon x^\prime \to y^\prime$ such that
\[
    r=(v,e,w v^\prime y^\prime w^\prime) \circ (vxw v^\prime,e,w^\prime) - (vyw v^\prime,e,w^\prime) \circ  (v,e,w v^\prime x^\prime w^\prime).
\]
By the definition of $\phi^\prime$, we have
\begin{align*}
    \phi^\prime([r])=&\phi^\prime(v,e,w v^\prime y^\prime w^\prime) \circ \phi^\prime(vxw v^\prime,e,w^\prime) -
    \phi^\prime(vyw v^\prime,e,w^\prime) \circ  \phi^\prime(v,e,w v^\prime x^\prime w^\prime)\\
    =& 1_{\phi(v)} \otimes \phi(e) \otimes 1_{\phi(w v^\prime y^\prime w^\prime)} \circ 1_{\phi(vxw v^\prime)} \otimes \phi(e^\prime) \otimes 1_{\phi(w^\prime)}  \\
    &-  1_{\phi(vyw v^\prime)} \otimes \phi(e^\prime) \otimes 1_{\phi(w^\prime)} \circ 1_{\phi(v)} \otimes \phi(e) \otimes 1_{\phi(w v^\prime x^\prime w^\prime)}
\end{align*}
Thus we have $\phi^\prime([r])=0$ by the interchange law in $\rD$.
With the well-definedness established, it's easy to show that $\phi^\prime$ is indeed a linear monoidal functor.
\end{proof}

Therefore, $\rM= \tM Q$, the free linear monoidal category over a monoidal quiver, always exists. It follows that one can always talk about presentations of a linear monoidal category.

\section{Presentation of linear monoidal categories as linear categories}\label{sec:MainTheorem}
In this section, we will give our main results: given a presentation of a linear monoidal category, we can produce a presentation of it as a linear category, and when all generating morphisms are endomorphisms, we can produce presentations of its endomorphism algebras. These results are in \cref{thm:presentationofalgebra,MainTheorem} and the former one is stated in \cite[\S 2.6]{BCNR17} without proof.

Before we get to the main theorem, we need the following lemma.
  \begin{lem}[The third isomorphism theorem]\label{ThirdIso} Let $\rL$ be a linear category and $\rI$, $\rJ$ be two ideals of $\rL$ such that $\rI$ is a subcategory of $\rJ$. Then $\rJ/\rI$ is also an ideal of $\rL/\rI$ and we have
    \begin{equation*}
        \rL/\rJ \cong (\rL/\rI)/(\rJ/\rI),
    \end{equation*}
where $\cong$ means isomorphic as linear categories.
\end{lem}
\begin{proof}
Consider a family of module morphisms $P_{a,b}$, $a,b \in \mathrm{Ob}(\rL)$: for each $f \in \rL(a,b)$, $P_{a,b}$ sends $[f]=f + \rI(a,b)$ to $f+ \rJ(a,b)$. It is well defined since $\rI$ is contained in $\rJ$. By the third isomorphism theorem of modules, we have there exists module isomorphism $\bar{P}_{a,b}$ from
\[ {( \rL/ \rI)/( \rJ/ \rI)}(a,b)=({ \rL/ \rI}(a,b))/({ \rJ/ \rI}(a,b))\]
to $\rL/\rJ(a,b)$. In particular, $\bar{P}_{a,b}$, sending each $[f]+\rJ/\rI(a,b)$ to $f + \rJ(a,b)$, is well-defined. Thus we can define a linear functor $\bar{P}$ induced by $\bar{P}_{a,b},\ a,b \in \mathrm{Ob}(\rL)$. The functor $\bar{P}$ does form a functor since we have
\begin{align*}
    P ([g] + \rJ/\rI(b,c)) ([f] +\rJ/\rI(a,b))= P([g \circ f] + \rJ/\rI(a,c))
    = g \circ f + \rJ(a,c) \\= (g + \rJ(b,c))(f+\rJ(a,b)) = P([g] + \rJ(b,c)) \circ P( [f] +\rJ(a,b)).
\end{align*}
Since $\bar{P}_{a,b}$ are linear isomorphisms, $\bar{P}$ is an isomorphism of linear categories.
\end{proof}

Now we reformulate and prove the theorem stated in \cite[\S 2.6]{Bru14} without proof.
\begin{theo}\label{MainTheorem}
  If a linear monoidal category $\rM$ has a presentation $(X,E,R)_\otimes$ as a linear monoidal category, then it has a presentation $(M(X),\bar{E},\bar{R})$ as a linear category, where $M(X),\bar{E},\bar{R}$ are defined as follows.
  \begin{itemize}
  \item $M(X)$ is the free monoid generated by $X$, where the product is denoted by concatenation.
  \item The set of generating morphisms $\bar{E}$ is
    \begin{align*}
      \bar{E}=\left\lbrace (v,e,w) \mid v,w \in M(X),\ e \in E \right\rbrace.
    \end{align*}
  \item The set of relations $\bar{R}$ is $C \cup R^\prime$, where $C$ and $R^\prime$ are defined as follows.
  \begin{itemize}
      \item The set $C$ is the relation of the interchange law. To be precise, $C$ consists of all morphisms of the form
\[
        (v,e,wv^\prime y^\prime w^\prime) \circ (vxwv^\prime,e^\prime,w^\prime) -
     (vywv^\prime,e^\prime,w^\prime) \circ  (v,e,wv^\prime x^\prime w^\prime),
\]
    where $v,v^\prime,w,w^\prime$ are objects and $e \colon x \to y, e^\prime \colon x^\prime \to y^\prime$ are edges in $E$.
  \item The set $R^\prime$ is constructed in the following way: for each relation $r$ in $R$ written as
      \begin{align*}
          [r]=\sum_{i=1}^n k_i [f_i],\quad
          [f_i]=\alpha^i_1 \circ \dotsb \circ \alpha^i_{n_i},\quad
          \alpha^i_j=(v^i_j, e^i_j,w^i_j),
      \end{align*}
  where $n$, $n_i \in \N$, $e^i_j \in E$,
  let $R^\prime$ include $r_{a,b}$, for $a$, $b \in M(X)$, where
      \begin{align*}
          r_{a,b}&= \sum_{i=1}^n k_i (f_i)_{a,b},\\
          (f_i)_{a,b}&=(\alpha^i_1)_{a,b} \circ \dotsb \circ (\alpha^i_{n_i})_{a,b},\\
          (\alpha^i_j)_{a,b}&=(a  v^i_j, e^i_j, w^i_j  b).
      \end{align*}
  \end{itemize}
  \end{itemize}
\end{theo}

\begin{rem}
In the construction of $R^\prime$, we choose some representatives of the cosets $[r]$, which might contradict well-definedness. However, we can see from the proof below that only the coset itself matters, namely the resulting presentation is independent with the choice of the representatives.
\end{rem}
\begin{proof}
	Let $Q=(X,E)$ and $\bar{Q}=(M(X),\bar{E})$ where $M(X)$ and $\bar{E}$ are defined as above. We shall prove that $\rM = \tM Q/ \left<R\right>_\otimes$, as a linear category, is isomorphic to $\tL \bar{Q} / \left<\bar{R}\right>$.
    Let $\rL= \tL \bar{Q}$, $\rI=\left<C\right>$ and $\rJ=\left<\bar{R}\right>$. Now $\rD = \rL/\rJ$ and by \cref{ThirdIso}, $\rD =(\rL/\rI)/(\rJ/\rI)$.
    Notice $\rI$ is exactly the interchange law for triples, so we can equip $\rL/\rI$ with the natural tensor product as in \cref{sec:LinMonCatExist}. Therefore $\rL/\rI$ coincides with $\tM Q$. Thus it suffices to show $\rJ/\rI$ is the same ideal as $\left<R\right>_\otimes$.

    On the one hand, by \cref{tenidealandidealstruc}, we have $\left<R\right>_\otimes$, as an ideal, is generated by $\tilde{R}$, where
\[
	\tilde{R}=\left\lbrace 1_a \otimes [r] \otimes 1_b \mid a,b \in M(X), r \in  R \right\rbrace  .
\]
    One the other hand, since $\rJ = \left<\bar{R}\right>= \left<C\cup R^\prime\right>$, the ideal $\rJ/\rI$ of $\rL/\rI=\tM Q$ is generated by
    \[
        [s], \quad s \in C \cup R^\prime.
    \]
   If $s \in C$, we have $[s]=[0]$. For $s \in R^\prime$, there exists $r \in R$ such that $s =r_{a,b}$. Namely
 \begin{align*}
        [s] &=  [r_{a,b}] = \sum_{i=1}^n k_i [(f_i)_{a,b}],\\
          [(f_i)_{a,b}] &= (\alpha^i_1)_{a,b} \circ \dotsb \circ (\alpha^i_{n_i})_{a,b},\\
          (\alpha^i_j)_{a,b}&=(a v^i_j, e^i_j, w^i_j b)= 1_{a} \otimes (v^i_j,e^i_j,w^i_j)  \otimes 1_{b},
 \end{align*}
for some $r$ such that
  \begin{align*}
          [r]=\sum_{i=1}^n k_i [f_i],\quad
          [f_i]=\alpha^i_1 \circ \dotsb \circ \alpha^i_{n_i},\quad
          \alpha^i_j=(v^i_j, e^i_j,w^i_j).
  \end{align*}
Since $(\alpha^i_j)_{a,b} = 1_{a} \otimes (v^i_j,e^i_j,w^i_j) \otimes 1_{b}$, we have $$[r_{a,b}]=1_a \otimes [r] \otimes 1_b.$$ (In particular, this implies that the choice of representatives doesn't matter.)
Thus $\rJ/\rI$ is also generated by $\tilde{R}$.
So we have $\rJ/\rI= \left<R\right>_\otimes.$
\end{proof}

Now we focus on the endomorphisms algebras. We consider a special case where all generators are endomorphisms.
\begin{theo}\label{thm:presentationofalgebra}
Let $\rD$ be a linear monoidal category with a presentation $\left<X,E,R\right>_\otimes$, where all generating morphisms in $E$ are endomorphisms and let $\left<M(X),\bar{E},\bar{R}\right>$ be its presentation as a linear category in \cref{MainTheorem}. For an object $a$, let $\bar{E}_a$ ($\bar{R}_a$)  be the set of  all generating morphisms in $\bar{E}$ (all the relations in $\bar{R}$) that are also in the endomorphism algebra $\End(a)$. Then the endomorphism algebra $\End(a)$ has an algebra presentation
\[
    \End(a) \cong \left< \bar{E}_a \mid \bar{R}_a \right>.
\]
\end{theo}
\begin{proof}
Since all morphisms in $E$ are endomorphisms, all morphisms in $\bar{E}$, namely
\[
    (1_v,e,1_w), \quad v,w \in M(X) \text{ and } e \in E,
\]
are endomorphisms. Therefore $\End_\rD(a)$, as a algebra, is generated by generators of $\rD$ which are also in $\End_\rD(a)$, denoted by $E_a$.
On the other hand, the ideal generated by $\bar{R}$, by \cref{idealstruc}, consists of morphisms of the form
\[
    	\sum_{i=1}^n k_i (f^\prime_i\circ r_i \circ f_i),
\]
where $k_i \in \kk$, $r_i \in R_a$ and $f_i, f^\prime_i \in \Hom(\rD)$.
Since all morphisms in $M$ are endomorphisms, its endomorphism algebra $\End_{\left<R\right>}(a)$ consists of morphisms of the above form with $f, f^\prime \in \End_\rD(a)$. Thus $\End_{\left<R\right>}(a)$ is exactly the ideal of the algebra $\End(a)$ generated by $\bar{R}_a$. Hence $\End(a)$ has the presentation $\left< \bar{E}_a \mid \bar{R}_a \right>$.
\end{proof}
\color{black}

\section{Monoidally generated algebras} \label{sec:MonGenAlg}

In this section, we will introduce some important (strict) linear monoidal categories by following the presentation in \cite[\S 3]{Savage18}, where all generating morphisms are endomorphisms, and apply \cref{thm:presentationofalgebra,MainTheorem} to produce presentations of their endomorphism algebras.

\subsection{String diagrams}\label{sec:StringDiag}
To help visualize these linear monoidal categories, we will use the notation of string diagrams. We follow the presentation of \cite[\S 2.2]{Savage18}.

A morphism $f \colon a \to b$ is denoted by a vertical strand with a coupon labeled $f$, read from the bottom to the top. In particular, identity morphisms are denoted by empty strands.
\[
  f \colon a \to b \ \text{is denoted by}\ \
  \begin{tikzpicture}[anchorbase]
    \draw (0,0) node[anchor=north] {\regionlabel{a}} to (0,1) node[anchor=south] {\regionlabel{b}};
    \filldraw[black,fill=white] (0,0.7) arc(90:450:0.2);
    \node at (0,0.5) {\tokenlabel{f}};
  \end{tikzpicture}, \qquad
   1_a \colon a \to a \ \text{is denoted by}\ \
  \begin{tikzpicture}[anchorbase]
    \draw (0,0) node[anchor=north] {\regionlabel{a}} to (0,1) node[anchor=south] {\regionlabel{a}};
  \end{tikzpicture}.
\]
Composition of morphisms is denoted by vertical stacking and tensor product is denoted by horizontal juxtaposition:
\[
  \begin{tikzpicture}[anchorbase]
    \draw (0,0) node[anchor=north] {\regionlabel{b}} to (0,1) node[anchor=south] {\regionlabel{c}};
    \filldraw[black,fill=white] (0,0.7) arc(90:450:0.2);
    \node at (0,0.5) {\tokenlabel{f}};
  \end{tikzpicture}
  \ \circ \
   \begin{tikzpicture}[anchorbase]
    \draw (0,0) node[anchor=north] {\regionlabel{a}} to (0,1) node[anchor=south] {\regionlabel{b}};
    \filldraw[black,fill=white] (0,0.7) arc(90:450:0.2);
    \node at (0,0.5) {\tokenlabel{g}};
  \end{tikzpicture}
        \ =\
  \begin{tikzpicture}[anchorbase]
       \draw (0,0) node[anchor=north] {\regionlabel{a}} to (0,1.4) node[anchor=south] {\regionlabel{c}};
    \filldraw[black,fill=white] (0,1.2) arc(90:450:0.2);
    \node at (0,1) {\tokenlabel{f}};
    \filldraw[black,fill=white] (0,0.6) arc(90:450:0.2);
    \node at (0,0.4) {\tokenlabel{g}};
  \end{tikzpicture}
  \ =\
   \begin{tikzpicture}[anchorbase]
    \draw (0,0) node[anchor=north] {\regionlabel{a}} to (0,1.4) node[anchor=south] {\regionlabel{c}};
    \filldraw[black,fill=white] (0,1) arc(90:450:0.3);
    \node at (0,0.7) {\tokenlabel{f \circ g}};
  \end{tikzpicture},
  \qquad
  \begin{tikzpicture}[anchorbase]
    \draw (0,0) node[anchor=north] {\regionlabel{a}} to (0,1) node[anchor=south] {\regionlabel{b}};
    \filldraw[black,fill=white] (0,0.7) arc(90:450:0.2);
    \node at (0,0.5) {\tokenlabel{f}};
  \end{tikzpicture}
  \ \otimes \
   \begin{tikzpicture}[anchorbase]
    \draw (0,0) node[anchor=north] {\regionlabel{c}} to (0,1) node[anchor=south] {\regionlabel{d}};
    \filldraw[black,fill=white] (0,0.7) arc(90:450:0.2);
    \node at (0,0.5) {\tokenlabel{g}};
  \end{tikzpicture}
  \ =\
  \begin{tikzpicture}[anchorbase]
    \draw (0,0) node[anchor=north] {\regionlabel{a}} to (0,1) node[anchor=south] {\regionlabel{b}};
    \filldraw[black,fill=white] (0,0.7) arc(90:450:0.2);
    \node at (0,0.5) {\tokenlabel{f}};
    \draw (0.5,0) node[anchor=north] {\regionlabel{c}} to (0.5,1) node[anchor=south] {\regionlabel{d}};
    \filldraw[black,fill=white] (0.5,0.7) arc(90:450:0.2);
    \node at (0.5,0.5) {\tokenlabel{g}};
  \end{tikzpicture}.
\]
The \emph{interchange law} for $f \colon a \to b$ and $g \colon c \to d$
\[
	(f \otimes 1_{d}) \circ (1_{a} \otimes g)
    \ =\
    f \otimes g
    \ =\
    (1_{b} \otimes g) \circ (f \otimes 1_{c})
\]
then becomes the following, where we omit the object labels:
\[
  \begin{tikzpicture}[anchorbase]
    \draw (0,0) to (0,1.6);
    \filldraw[black,fill=white] (0,1.3) arc(90:450:0.2);
    \node at (0,1.1) {\tokenlabel{f}};
    \draw (0.5,0) to (0.5,1.6);
    \filldraw[black,fill=white] (0.5,0.7) arc(90:450:0.2);
    \node at (0.5,0.5) {\tokenlabel{g}};
  \end{tikzpicture}
  \ =\
  \begin{tikzpicture}[anchorbase]
    \draw (0,0) to (0,1.6);
    \filldraw[black,fill=white] (0,1) arc(90:450:0.2);
    \node at (0,0.8) {\tokenlabel{f}};
    \draw (0.5,0) to (0.5,1.6);
    \filldraw[black,fill=white] (0.5,1) arc(90:450:0.2);
    \node at (0.5,0.8) {\tokenlabel{g}};
  \end{tikzpicture}
  \ =\
  \begin{tikzpicture}[anchorbase]
    \draw (0,0) to (0,1.6);
    \filldraw[black,fill=white] (0,0.7) arc(90:450:0.2);
    \node at (0,0.5) {\tokenlabel{f}};
    \draw (0.5,0) to (0.5,1.6);
    \filldraw[black,fill=white] (0.5,1.3) arc(90:450:0.2);
    \node at (0.5,1.1) {\tokenlabel{g}};
  \end{tikzpicture}.
\]

In the following sections we will use string diagrams to visualize the generators and relations of several linear monoidal categories. We may generalize the horizontal juxaposition by using it to denote tensor products of objects and triples $(v,e,w)$, namely
\[
    1_{v \otimes w} =
    \begin{tikzpicture}[anchorbase]
    \draw  (0,0) node[anchor=north] {\regionlabel{v}}
    to (0,1) node[anchor=south] {\regionlabel{v}} ;
    \draw   (0.4,0) node[anchor=north] {\regionlabel{w}}
    to (0.4,1) node[anchor=south] {\regionlabel{w}} ;
    \end{tikzpicture}
    \quad \text{and} \quad
    (v,e,w)
    \ =\
	\begin{tikzpicture}[anchorbase]
    \draw (-0.5,0) node[anchor=north] {\regionlabel{v}} to (-0.5,1)
    node[anchor=south] {\regionlabel{v}};
    \draw (-0,-0) node[anchor=north] {\regionlabel{x}} to (0,1) node[anchor=south] {\regionlabel{y}};
    \filldraw[black,fill=white] (0,0.7) arc(90:450:0.2);
    \node at (0,0.5) {\tokenlabel{e}};
    \draw (0.5,0) node[anchor=north] {\regionlabel{w}} to (0.5,1) node[anchor=south] {\regionlabel{w}};
  \end{tikzpicture}
  \quad \text{for } e\colon x \to y.
\]

\color{black}
\subsection{The Symmetric group} \label{sec:Sym}
Let $\cS$ be the linear monoidal category with the presentation $\left<X_\cS,E_\cS,R_\cS\right>_\otimes$, where
\[
    X_\cS=\left\lbrace a\right\rbrace  ,\  E_\cS=\left\lbrace e \colon a\otimes a \to a \otimes a\right\rbrace  \text{ and }R_\cS=\left\lbrace r_1,r_2\right\rbrace  ,
\]
where
\[
	r_1 = e^2-1_a,\quad
    r_2 = (e\otimes 1_a) \circ (1_a \otimes e) \circ (e \otimes 1_a) - (1_a \otimes e) \circ (e \otimes 1_a) \circ (1_a \otimes e).
\]
Using string diagrams, we denote
\begin{itemize}
    \item the generating object $a$ as $\uparrow$ and,
    \item the generating morphism $e$ as
$
      \begin{tikzpicture}[anchorbase]
        \draw[->] (-0.25,-0.25) to (0.25,0.25);
        \draw[->] (0.25,-0.25) to (-0.25,0.25);
      \end{tikzpicture},
$
\end{itemize}
then $R_\cS$ becomes
\begin{equation} \label{Sn-strings}
      \begin{tikzpicture}[anchorbase]
        \draw[->] (0.3,0) to[out=up,in=down] (-0.1,0.5) to[out=up,in=down] (0.3,1);
        \draw[->] (-0.1,0) to[out=up,in=down] (0.3,0.5) to[out=up,in=down] (-0.1,1);
      \end{tikzpicture}
      \ =\
      \begin{tikzpicture}[anchorbase]
        \draw[->] (-0.2,0) -- (-0.2,1);
        \draw[->] (0.2,0) -- (0.2,1);
      \end{tikzpicture}
            \qquad \text{and} \qquad
      \begin{tikzpicture}[anchorbase]
        \draw[->] (0.4,0) -- (-0.4,1);
        \draw[->] (0,0) to[out=up, in=down] (-0.4,0.5) to[out=up,in=down] (0,1);
        \draw[->] (-0.4,0) -- (0.4,1);
      \end{tikzpicture}
      \ =\
      \begin{tikzpicture}[anchorbase]
        \draw[->] (0.4,0) to  (-0.4,1);
        \draw[->] (0,0) to[out=up, in=down] (0.4,0.5) to[out=up,in=down] (0,1);
        \draw[->] (-0.4,0) to  (0.4,1);
      \end{tikzpicture}
      \ .
\end{equation}

By \cref{MainTheorem} 
, we have that $\cS$ has the presentation $\left<M(X_\cS),\overline{E_\cS},\overline{R_\cS}\right>$ as a linear category.
\begin{itemize}
\item The objects are $M(X_\cS)=\left\lbrace \uparrow^{\otimes n} \mid n \in \N\right\rbrace  $. To simplify our notation, we also denote the tensor product on objects $\otimes$ as concatenation. Thus  $M(X_\cS)=\left\lbrace \uparrow^{ n} \mid n \in \N\right\rbrace  $
\item The generating morphisms in $\overline{E_\cS}$ are $(v,e,w)$, where $e \colon \uparrow^2 \to \uparrow^2$ and $v,w \in M(X_\cS)$, namely
\[
    \bar{E}_\cS = \left\lbrace \
    \left(\uparrow^m,
      \begin{tikzpicture}[anchorbase]
        \draw[->] (-0.25,-0.25) to (0.25,0.25);
        \draw[->] (0.25,-0.25) to (-0.25,0.25);
      \end{tikzpicture},
     \uparrow^n\right)
    \ =\
    \begin{tikzpicture}[anchorbase]
	    \draw[->] (-1.2,0) to (-1.2,0.5);
	    \node at (-0.8,0.25) {$\cdots$};
	    \node at (-0.8,0) {\regionlabel{m}};
	    \draw[->] (-0.4,0) to (-0.4,0.5);
	    \draw[->] (0,0) to [out=90, in=270]  (0.4,0.5);
	    \draw[->] (0.4,0) to [out=90, in=270]  (0,0.5);
	    \draw[->] (0.8,0) to (0.8,0.5);
	    \node at (1.2,0.25) {$\cdots$};
	    \node at (1.2,0) {\regionlabel{n}};
	    \draw[->] (1.6,0) to (1.6,0.5);
	\end{tikzpicture}\
	\mid
	m,n \in \N \right\rbrace  .
\]
\item The relations are $\overline{R_\cS} = C_\cS \cup R_\cS^\prime$. The relations in $C_\cS$ are the interchange laws of triples, namely
\begin{align} \label{SymInt}
    \begin{tikzpicture}[anchorbase]
	    \draw[->] (-1.2,0) to (-1.2,1);
	    \node at (-0.8,0.5) {$\cdots$};
	    \node at (-0.8,0) {\regionlabel{m}};
	    \draw[->] (-0.4,0) to (-0.4,1);
	    \draw[->] (0,0) to (0,0.5) to [out=90, in=270]  (0.4,1);
	    \draw[->] (0.4,0) to (0.4,0.5) to [out=90, in=270]  (0,1);
	    \draw[->] (0.8,0) to (0.8,1);
	    \node at (1.2,0.5) {$\cdots$};
	    \node at (1.2,0) {\regionlabel{l}};
	    \draw[->] (1.6,0) to (1.6,1);
	    \draw[->] (2,0) to [out=90, in=270]  (2.4,0.5) to (2.4,1);
	    \draw[->] (2.4,0) to [out=90, in=270]  (2,0.5) to (2,1);
	    \draw[->] (2.8,0) to (2.8,1);
	    \node at (3.2,0.5) {$\cdots$};
	    \node at (3.2,0) {\regionlabel{n}};
	    \draw[->] (3.6,0) to (3.6,1);
	\end{tikzpicture}
	\ =\
	    \begin{tikzpicture}[anchorbase]
	    \draw[->] (-1.2,0) to (-1.2,1);
	    \node at (-0.8,0.5) {$\cdots$};
	    \node at (-0.8,0) {\regionlabel{m}};
	    \draw[->] (-0.4,0) to (-0.4,1);
	    \draw[->] (0,0) to [out=90, in=270] (0.4,0.5) to   (0.4,1);
	    \draw[->] (0.4,0) to  [out=90, in=270]  (0,0.5) to (0,1);
	    \draw[->] (0.8,0) to (0.8,1);
	    \node at (1.2,0.5) {$\cdots$};
	    \node at (1.2,0) {\regionlabel{l}};
	    \draw[->] (1.6,0) to (1.6,1);
	    \draw[->] (2,0) to  (2.0,0.5) to  [out=90, in=270] (2.4,1);
	    \draw[->] (2.4,0) to  (2.4 ,0.5) to [out=90, in=270] (2,1);
	    \draw[->] (2.8,0) to (2.8,1);
	    \node at (3.2,0.5) {$\cdots$};
	    \node at (3.2,0) {\regionlabel{n}};
	    \draw[->] (3.6,0) to (3.6,1);
	\end{tikzpicture},
\end{align}
for all $m,n,l \in \N$.
The relations in $R_\cS^\prime$ are $(r_i)_{a,b}$ as defined in \cref{MainTheorem}, where $r_i \in R$, $a,b \in M(X)$, namely
\begin{align} \label{SymRel}
\begin{split}
    \begin{tikzpicture}[anchorbase]
	    \draw[->] (-1.2,0) to (-1.2,1);
	    \node at (-0.8,0.5) {$\cdots$};
	    \node at (-0.8,0) {\regionlabel{m}};
	    \draw[->] (-0.4,0) to (-0.4,1);
	    \draw[->] (0,0) to [out=90, in=270]  (0.4,0.5) to [out=up, in=down] (0,1);
	    \draw[->] (0.4,0) to [out=90, in=270] (0,0.5) to [out=up,in=down] (0.4,1);
	    \draw[->] (0.8,0) to (0.8,1);
	    \node at (1.2,0.5) {$\cdots$};
	    \node at (1.2,0) {\regionlabel{n}};
	    \draw[->] (1.6,0) to (1.6,1);
	\end{tikzpicture}
	\ &=\
	\begin{tikzpicture}[anchorbase]
	    \draw[->] (-1.2,0) to (-1.2,1);
	    \node at (-0.8,0.5) {$\cdots$};
	    \node at (-0.8,0) {\regionlabel{m}};
	    \draw[->] (-0.4,0) to (-0.4,1);
	    \draw[->] (0,0) to (0,1);
	    \draw[->] (0.4,0) to (0.4,1);
	    \draw[->] (0.8,0) to (0.8,1);
	    \node at (1.2,0.5) {$\cdots$};
	    \node at (1.2,0) {\regionlabel{n}};
	    \draw[->] (1.6,0) to (1.6,1);
	\end{tikzpicture}
	,\\
    \begin{tikzpicture}[anchorbase]
	    \draw[->] (-1.2,0) to (-1.2,1);
	    \node at (-0.8,0.5) {$\cdots$};
	    \node at (-0.8,0) {\regionlabel{m}};
	    \draw[->] (-0.4,0) to (-0.4,1);
	    \draw[->] (0.8,0) -- (0,1);
        \draw[->] (0.4,0) to[out=up, in=down] (0,0.5) to[out=up,in=down] (0.4,1);
        \draw[->] (0,0) -- (0.8,1);
	    \draw[->] (1.2,0) to (1.2,1);
	    \node at (1.6,0.5) {$\cdots$};
	    \node at (1.6,0) {\regionlabel{n}};
	    \draw[->] (2,0) to (2,1);
	\end{tikzpicture}
	\ &=\
    \begin{tikzpicture}[anchorbase]
	    \draw[->] (-1.2,0) to (-1.2,1);
	    \node at (-0.8,0.5) {$\cdots$};
	    \node at (-0.8,0) {\regionlabel{m}};
	    \draw[->] (-0.4,0) to (-0.4,1);
	   \draw[->] (0.8,0) to  (0,1);
        \draw[->] (0.4,0) to[out=up, in=down] (0.8,0.5) to[out=up,in=down] (0.4,1);
        \draw[->] (0,0) to  (0.8,1);	
        \draw[->] (1.2,0) to (1.2,1);
	    \node at (1.6,0.5) {$\cdots$};
	    \node at (1.6,0) {\regionlabel{n}};
	    \draw[->] (2,0) to (2,1);
	\end{tikzpicture},
\end{split}
\end{align}
for all $m,n \in \N$.
\end{itemize}

Fix $d \in \N$, and consider the endomorphism algebra $\End(\uparrow^d)$.
Since the only generating morphism is an endomorphism, by \cref{thm:presentationofalgebra}, we have that
generators of $\End(\uparrow^d)$ are
\begin{align}\label{sym-end-gen}
s_i =
\begin{tikzpicture}[anchorbase]
	    \draw[->] (-1.2,0) to (-1.2,0.5);
	    \node at (-0.8,0.25) {$\cdots$};
	    \node[below] at (-0.8,0) {\regionlabel{d-1-i}};
	    \draw[->] (-0.4,0) to (-0.4,0.5);
	    \draw[->] (0,0) to [out=90, in=270]  (0.4,0.5);
	    \draw[->] (0.4,0) to [out=90, in=270]  (0,0.5);
	    \draw[->] (0.8,0) to (0.8,0.5);
	    \node at (1.2,0.25) {$\cdots$};
	    \node[below] at (1.2,0) {\regionlabel{i-1}};
	    \draw[->] (1.6,0) to (1.6,0.5);
	\end{tikzpicture},
\quad i=1,...,d-1.
\end{align}
Then relations that the generators $s_i,i=1,...,d-1$ satisfy, denoted by $\overline{R_\cS}^d$, are:
\begin{align}\label{sym-end-rel}
\begin{split}
	s_i s_j = s_j s_i,  \quad&  \left|i-j\right|>1,\\
    s_i^2 = 1,  \quad& i=1,\dots,d-1,\\
  	s_{i+1} s_i s_{i+1} = s_i s_{i+1} s_i, \quad & i=1,\dots,d-2,
\end{split}
\end{align}
where the first one is from $C_\cS$
and the last two are from $R^\prime_\cS$
By \cref{thm:presentationofalgebra}, we have that $\End_\cS(\uparrow^d)$, as an algebra, has the presentation $\left< s_1,\dots,s_{d-1} \mid \overline{R_\cS}^d \right>$. Notice this is exactly the presentation of the group algebra of symmetric group $S_d$. Thus we have
\begin{align*}
	\End_\cS(\uparrow^d) \cong \kk S_d.
\end{align*}

\subsection{Degenerate affine Hecke algebras\label{sec:dAHA}}

Let $\AH^\dg$ be the strict $\kk$-linear monoidal category $\cS$ defined in \cref{sec:Sym}, but with an additional generating morphism
$
  \begin{tikzpicture}[anchorbase]
    \draw[->] (0,0) to (0,0.6);
    \redcircle{(0,0.3)};
  \end{tikzpicture}
  \ \colon \uparrow\ \to\ \uparrow
$
and one additional relation:
\[
  \begin{tikzpicture}[anchorbase]
    \draw[->] (0,0) -- (0.6,0.6);
    \draw[->] (0.6,0) -- (0,0.6);
    \redcircle{(0.15,.45)};
  \end{tikzpicture}
  \ -\
  \begin{tikzpicture}[anchorbase]
    \draw[->] (0,0) -- (0.6,0.6);
    \draw[->] (0.6,0) -- (0,0.6);
    \redcircle{(.45,.15)};
  \end{tikzpicture}
  \ =\
  \begin{tikzpicture}[anchorbase]
    \draw[->] (0,0) -- (0,0.6);
    \draw[->] (0.3,0) -- (0.3,0.6);
  \end{tikzpicture}\ .
\]
It follows that all morphisms in $\AH^\dg$ are endomoprhisms, and we have that \cref{thm:presentationofalgebra} applies.
Thus if we fix $d$, we have that the endormorphism algebra $\End(\uparrow^d)$ is generated by
\[
    s_i =
  \begin{tikzpicture}[anchorbase]
	    \draw[->] (-1.2,0) to (-1.2,0.5);
	    \node at (-0.8,0.25) {$\cdots$};
	    \node[below] at (-0.8,0) {\regionlabel{d-1-i}};
	    \draw[->] (-0.4,0) to (-0.4,0.5);
	    \draw[->] (0,0) to [out=90, in=270]  (0.4,0.5);
	    \draw[->] (0.4,0) to [out=90, in=270]  (0,0.5);
	    \draw[->] (0.8,0) to (0.8,0.5);
	    \node at (1.2,0.25) {$\cdots$};
	    \node[below] at (1.2,0) {\regionlabel{i-1}};
	    \draw[->] (1.6,0) to (1.6,0.5);
	\end{tikzpicture},
	\quad \text{and} \quad
	t_j =
	  \begin{tikzpicture}[anchorbase]
	    \draw[->] (-1.2,0) to (-1.2,0.5);
	    \node at (-0.8,0.25) {$\cdots$};
	    \node[below] at (-0.8,0) {\regionlabel{d-j}};
	    \draw[->] (-0.4,0) to (-0.4,0.5);
	    \draw[->] (0,0) to (0,0.5);
	    \redcircle{(0,0.25)};
	    \draw[->] (0.4,0) to (0.4,0.5);
	    \node at (0.8,0.25) {$\cdots$};
	    \node[below] at (0.8,0) {\regionlabel{j-1}};
	    \draw[->] (1.2,0) to (1.2,0.5);
	\end{tikzpicture},
\]
where $i=1,\dots,d-1$ and $j=1,\dots,d$. The generators $s_i$ satisfy the relations in \cref{SymRel} and \cref{SymInt}. Besides, the generators $s_i$ and $t_j$ satisfy that:
\begin{align}\label{ahdeg-end-rel}
\begin{split}
     \begin{tikzpicture}[anchorbase]
	    \draw[->] (-1.2,0) to (-1.2,1);
	    \node at (-0.8,0.5) {$\cdots$};
	    \node[below] at (-0.8,0) {\regionlabel{d-1-i}};
	    \draw[->] (-0.4,0) to (-0.4,1);
	    \draw[->] (0,0) to [out=90, in=270] (0.4,0.5) to   (0.4,1);
	    \draw[->] (0.4,0) to  [out=90, in=270]  (0,0.5) to (0,1);
	    \redcircle{(0,0.75)};
	    \draw[->] (0.8,0) to (0.8,1);
	    \node at (1.2,0.5) {$\cdots$};
	    \node[below] at (1.2,0) {\regionlabel{i-1}};
	    \draw[->] (1.6,0) to (1.6,1);
    \end{tikzpicture}
    \ -\
    \begin{tikzpicture}[anchorbase]
	    \draw[->] (-1.2,0) to (-1.2,1);
	    \node at (-0.8,0.5) {$\cdots$};
	    \node[below] at (-0.8,0) {\regionlabel{d-1-i}};
	    \draw[->] (-0.4,0) to (-0.4,1);
	    \draw[->] (0,0) to (0,0.5) to [out=90, in=270]  (0.4,1);
	    \draw[->] (0.4,0) to (0.4,0.5) to [out=90, in=270]  (0,1);
	    \redcircle{(0.4,0.25)};
	    \draw[->] (0.8,0) to (0.8,1);
	    \node at (1.2,0.5) {$\cdots$};
	    \node[below] at (1.2,0) {\regionlabel{i-1}};
	    \draw[->] (1.6,0) to (1.6,1);
	\end{tikzpicture}
	\ =\
	 \begin{tikzpicture}[anchorbase]
	    \draw[->] (-1.2,0) to (-1.2,1);
	    \node at (-0.8,0.5) {$\cdots$};
	    \node[below] at (-0.8,0) {\regionlabel{d-1-i}};
	    \draw[->] (-0.4,0) to (-0.4,1);
	    \draw[->] (0,0) to (0,1);
	    \draw[->] (0.4,0) to (0.4,1);
	    \draw[->] (0.8,0) to (0.8,1);
	    \node at (1.2,0.5) {$\cdots$};
	    \node[below] at (1.2,0) {\regionlabel{i-1}};
	    \draw[->] (1.6,0) to (1.6,1);
	\end{tikzpicture},
\end{split}
\end{align}
\begin{align}\label{ahdeg-end-interchange}
\begin{split}
    \begin{tikzpicture}[anchorbase]
	    \draw[->] (-1.2,0) to (-1.2,1);
	    \node at (-0.8,0.5) {$\cdots$};
	    \node[below] at (-0.8,0) {\regionlabel{d-j}};
	    \draw[->] (-0.4,0) to (-0.4,1);
	    \draw[->] (0,0) to (0,1);
	    \redcircle{(0,0.75)};
	    \draw[->] (0.4,0) to (0.4,1);
	    \node at (0.8,0.5) {$\cdots$};
	    \node[below] at (0.8,0) {\regionlabel{(j-i)-1}};
	    \draw[->] (1.2,0) to (1.2,1);
	  	\draw[->] (1.6,0) to (1.6,1);
	    \redcircle{(1.6,0.25)};
	    \draw[->] (2,0) to (2, 1);
	    \node at (2.4,0.5) {$\cdots$};
	    \node[below] at (2.4,0) {\regionlabel{i-1}};
	    \draw[->] (2.8,0) to (2.8,1);
	\end{tikzpicture}
	\ =&\
	 \begin{tikzpicture}[anchorbase]
	    \draw[->] (-1.2,0) to (-1.2,1);
	    \node at (-0.8,0.5) {$\cdots$};
	    \node[below] at (-0.8,0) {\regionlabel{d-j}};
	    \draw[->] (-0.4,0) to (-0.4,1);
	    \draw[->] (0,0) to (0,1);
	    \redcircle{(0,0.25)};
	    \draw[->] (0.4,0) to (0.4,1);
	    \node at (0.8,0.5) {$\cdots$};
	    \node[below] at (0.8,0) {\regionlabel{(j-i)-1}};
	    \draw[->] (1.2,0) to (1.2,1);
	  	\draw[->] (1.6,0) to (1.6,1);
	    \redcircle{(1.6,0.75)};
	    \draw[->] (2,0) to (2, 1);
	    \node at (2.4,0.5) {$\cdots$};
	    \node[below] at (2.4,0) {\regionlabel{i-1}};
	    \draw[->] (2.8,0) to (2.8,1);
	\end{tikzpicture},
	\\
	    \begin{tikzpicture}[anchorbase]
	    \draw[->] (-1.6,0) to (-1.6,1);
	    \node at (-1.2,0.5) {$\cdots$};
	    \node[below] at (-1.2,0) {\regionlabel{d-1-j}};
	    \draw[->] (-0.8,0) to (-0.8,1);
        \draw[->] (-0.4,0) to [out=90, in=270] (0,0.5) to  (0,1);
	    \draw[->] (0,0) to  [out=90, in=270]  (-0.4,0.5) to (-0.4,1);
        \draw[->] (0.4,0) to (0.4,1);
	    \node at (0.8,0.5) {$\cdots$};
	    \node[below] at (0.8,0) {\regionlabel{(j-i)-1}};
	    \draw[->] (1.2,0) to (1.2,1);
	  	\draw[->] (1.6,0) to (1.6,1);
	    \redcircle{(1.6,0.75)};
	    \draw[->] (2,0) to (2, 1);
	    \node at (2.4,0.5) {$\cdots$};
	    \node[below] at (2.4,0) {\regionlabel{i-1}};
	    \draw[->] (2.8,0) to (2.8,1);
	\end{tikzpicture}
	\ =&\
		\begin{tikzpicture}[anchorbase]
	    \draw[->] (-1.6,0) to (-1.6,1);
	    \node at (-1.2,0.5) {$\cdots$};
	    \node[below] at (-1.2,0) {\regionlabel{d-1-j}};
	    \draw[->] (-0.8,0) to (-0.8,1);
        \draw[->] (-0.4,0) to  (-0.4,0.5) to [out=90, in=270] (0,1);
	    \draw[->] (0,0) to (0,0.5) to  [out=90, in=270]  (-0.4,1);
        \draw[->] (0.4,0) to (0.4,1);
	    \node at (0.8,0.5) {$\cdots$};
	    \node[below] at (0.8,0) {\regionlabel{(j-i)-1}};
	    \draw[->] (1.2,0) to (1.2,1);
	  	\draw[->] (1.6,0) to (1.6,1);
	    \redcircle{(1.6,0.25)};
	    \draw[->] (2,0) to (2, 1);
	    \node at (2.4,0.5) {$\cdots$};
	    \node[below] at (2.4,0) {\regionlabel{i-1}};
	    \draw[->] (2.8,0) to (2.8,1);
	\end{tikzpicture},
\end{split}
\end{align}
where $i,j=1,\dots,d-1,$ and in the last relation $i\neq j, i\neq j+1$.
The generators and relations form the presentation of the \emph{degenerate affine Hecke algebra} of type $A_{d-1}$. One can refer to \cite[5.55]{Mol00} for more details about these algebras.

\color{black}

\subsection{The braid group} \label{braid-cat}
Let $\cB$ be a linear monoidal category with the presentation
$ \left<
    \uparrow,E_\cB,R_\cB
 \right>_\otimes,$
 where the generating morphisms in $E_\cB$ are
\begin{equation} \label{poscross}
 \begin{tikzpicture}[anchorbase]
 	\draw[->]  (0.25, -0.25) to (-0.25,0.25);
   	\draw[wipe] (-0.25,-0.25) to (0.25,0.25);
    \draw[->] (-0.25,-0.25) to (0.25,0.25);
  \end{tikzpicture}
  \quad \text{and} \quad
  \begin{tikzpicture}[anchorbase]
 	\draw[->]  (-0.25, -0.25) to (0.25,0.25);
   	\draw[wipe] (0.25,-0.25) to (-0.25,0.25);
    \draw[->] (0.25,-0.25) to (-0.25,0.25);
  \end{tikzpicture},
\end{equation}
and relations in $R_\cB$ are
\begin{equation} \label{braid-rel}
  \begin{tikzpicture}[anchorbase]
    \draw[->] (0.3,0) to[out=up,in=down] (-0.1,0.5) to[out=up,in=down] (0.3,1);
    \draw[wipe] (-0.1,0) to[out=up,in=down] (0.3,0.5) to[out=up,in=down] (-0.3,1);
    \draw[->] (-0.1,0) to[out=up,in=down] (0.3,0.5) to[out=up,in=down] (-0.1,1);
  \end{tikzpicture}
  \ =\
  \begin{tikzpicture}[anchorbase]
    \draw[->] (-0.2,0) -- (-0.2,1);
    \draw[->] (0.2,0) -- (0.2,1);
  \end{tikzpicture},
  \quad
  \begin{tikzpicture}[anchorbase]
    \draw[->] (-0.1,0) to[out=up,in=down] (0.3,0.5) to[out=up,in=down] (-0.1,1);
    \draw[wipe] (0.3,0) to[out=up,in=down] (-0.1,0.5) to[out=up,in=down] (0.3,1);
    \draw[->] (0.3,0) to[out=up,in=down] (-0.1,0.5) to[out=up,in=down] (0.3,1);
  \end{tikzpicture}
  \ =\
  \begin{tikzpicture}[anchorbase]
    \draw[->] (-0.2,0) -- (-0.2,1);
    \draw[->] (0.2,0) -- (0.2,1);
  \end{tikzpicture}
  \quad \text{and} \quad
  \begin{tikzpicture}[anchorbase]
    \draw[->] (0.4,0) -- (-0.4,1);
    \draw[wipe] (0,0) to[out=up, in=down] (-0.4,0.5) to[out=up,in=down] (0,1);
    \draw[->] (0,0) to[out=up, in=down] (-0.4,0.5) to[out=up,in=down] (0,1);
    \draw[wipe] (-0.4,0) -- (0.4,1);
    \draw[->] (-0.4,0) -- (0.4,1);
  \end{tikzpicture}
  \ =\
  \begin{tikzpicture}[anchorbase]
    \draw[->] (0.4,0) -- (-0.4,1);
    \draw[wipe] (0,0) to[out=up, in=down] (0.4,0.5) to[out=up,in=down] (0,1);
    \draw[->] (0,0) to[out=up, in=down] (0.4,0.5) to[out=up,in=down] (0,1);
    \draw[wipe] (-0.4,0) -- (0.4,1);
    \draw[->] (-0.4,0) -- (0.4,1);
  \end{tikzpicture}
  \ .
\end{equation}
The first two relations in \cref{braid-rel} actually mean that the two generating morphisms are inverses of each other, so it suffices to give only the first one generating morphism and assume it is invertible.
By \cref{MainTheorem}, we obtain that $\cB$ has a presentation $\left< M(\uparrow), \overline{E_\cB},  \overline{R_\cB}\right>$ as a linear category, where:
\begin{itemize}
	\item $M(\uparrow)=\left\lbrace \ \uparrow^n \ \mid n \in \N\ \right\rbrace  $,
    \item Morphisms in $\overline{E_\cB}$ are
    \begin{equation}\label{braid-generator}
    b_{n,m}=
    \begin{tikzpicture}[anchorbase]
	    \draw[->] (-1.2,0) to (-1.2,0.5);
	    \node at (-0.8,0.25) {$\cdots$};
	    \node at (-0.8,0) {\regionlabel{m}};
	    \draw[->] (-0.4,0) to (-0.4,0.5);
	    \draw[->] (0.4,0) to [out=90, in=270]  (0,0.5);
	    \draw[wipe] (0,0) to [out=90, in=270]  (0.4,0.5);
	    \draw[->] (0,0) to [out=90, in=270]  (0.4,0.5);
	    \draw[->] (0.8,0) to (0.8,0.5);
	    \node at (1.2,0.25) {$\cdots$};
	    \node at (1.2,0) {\regionlabel{n}};
	    \draw[->] (1.6,0) to (1.6,0.5);
	\end{tikzpicture}
	\quad \text{and its inverse} \quad
	    \begin{tikzpicture}[anchorbase]
	    \draw[->] (-1.2,0) to (-1.2,0.5);
	    \node at (-0.8,0.25) {$\cdots$};
	    \node at (-0.8,0) {\regionlabel{m}};
	    \draw[->] (-0.4,0) to (-0.4,0.5);
	    \draw[->] (0,0) to [out=90, in=270]  (0.4,0.5);
	    \draw[wipe] (0.4,0) to [out=90, in=270]  (0,0.5);
	    \draw[->] (0.4,0) to [out=90, in=270]  (0,0.5);
	    \draw[->] (0.8,0) to (0.8,0.5);
	    \node at (1.2,0.25) {$\cdots$};
	    \node at (1.2,0) {\regionlabel{n}};
	    \draw[->] (1.6,0) to (1.6,0.5);
	\end{tikzpicture},
\end{equation}
where $m,n \in \N$
    \item Relations in $\overline{R_\cB}$ are $C_\cB$ and ${R_\cB}^\prime$, where ${R_\cB}^\prime$ consists of morphisms
\begin{align}
    \begin{split}
    \begin{tikzpicture}[anchorbase]
	    \draw[->] (-1.2,0) to (-1.2,1);
	    \node at (-0.8,0.5) {$\cdots$};
	    \node at (-0.8,0) {\regionlabel{m}};
	    \draw[->] (-0.4,0) to (-0.4,1);
	     \draw[->] (0.4,0) to [out=90, in=270] (0,0.5) to [out=up,in=down] (0.4,1);
	    \draw[wipe] (0,0) to [out=90, in=270]  (0.4,0.5) to [out=up, in=down] (0,1);
	    \draw[->] (0,0) to [out=90, in=270]  (0.4,0.5) to [out=up, in=down] (0,1);
	    \draw[->] (0.8,0) to (0.8,1);
	    \node at (1.2,0.5) {$\cdots$};
	    \node at (1.2,0) {\regionlabel{n}};
	    \draw[->] (1.6,0) to (1.6,1);
	\end{tikzpicture}
	\ &=\
	\begin{tikzpicture}[anchorbase]
	    \draw[->] (-1.2,0) to (-1.2,1);
	    \node at (-0.8,0.5) {$\cdots$};
	    \node at (-0.8,0) {\regionlabel{m}};
	    \draw[->] (-0.4,0) to (-0.4,1);
	    \draw[->] (0,0) to (0,1);
	    \draw[->] (0.4,0) to (0.4,1);
	    \draw[->] (0.8,0) to (0.8,1);
	    \node at (1.2,0.5) {$\cdots$};
	    \node at (1.2,0) {\regionlabel{n}};
	    \draw[->] (1.6,0) to (1.6,1);
	\end{tikzpicture}
	,\\
	\begin{tikzpicture}[anchorbase]
	    \draw[->] (-1.2,0) to (-1.2,1);
	    \node at (-0.8,0.5) {$\cdots$};
	    \node at (-0.8,0) {\regionlabel{m}};
	    \draw[->] (-0.4,0) to (-0.4,1);
	    \draw[->] (0,0) to [out=90, in=270]  (0.4,0.5) to [out=up, in=down] (0,1);
	    \draw[wipe] (0.4,0) to [out=90, in=270] (0,0.5) to [out=up,in=down] (0.4,1);
	    \draw[->] (0.4,0) to [out=90, in=270] (0,0.5) to [out=up,in=down] (0.4,1);
	    \draw[->] (0.8,0) to (0.8,1);
	    \node at (1.2,0.5) {$\cdots$};
	    \node at (1.2,0) {\regionlabel{n}};
	    \draw[->] (1.6,0) to (1.6,1);
	\end{tikzpicture}
	\ &=\
	\begin{tikzpicture}[anchorbase]
	    \draw[->] (-1.2,0) to (-1.2,1);
	    \node at (-0.8,0.5) {$\cdots$};
	    \node at (-0.8,0) {\regionlabel{m}};
	    \draw[->] (-0.4,0) to (-0.4,1);
	    \draw[->] (0,0) to (0,1);
	    \draw[->] (0.4,0) to (0.4,1);
	    \draw[->] (0.8,0) to (0.8,1);
	    \node at (1.2,0.5) {$\cdots$};
	    \node at (1.2,0) {\regionlabel{n}};
	    \draw[->] (1.6,0) to (1.6,1);
	\end{tikzpicture}	
	,\\
    \begin{tikzpicture}[anchorbase]
	    \draw[->] (-1.2,0) to (-1.2,1);
	    \node at (-0.8,0.5) {$\cdots$};
	    \node at (-0.8,0) {\regionlabel{m}};
	    \draw[->] (-0.4,0) to (-0.4,1);
	    \draw[->] (0.8,0) -- (0,1);
	    \draw[wipe] (0.4,0) to[out=up, in=down] (0,0.5) to[out=up,in=down] (0.4,1);
        \draw[->] (0.4,0) to[out=up, in=down] (0,0.5) to[out=up,in=down] (0.4,1);
        \draw[wipe] (0,0) -- (0.8,1);
        \draw[->] (0,0) -- (0.8,1);
	    \draw[->] (1.2,0) to (1.2,1);
	    \node at (1.6,0.5) {$\cdots$};
	    \node at (1.6,0) {\regionlabel{n}};
	    \draw[->] (2,0) to (2,1);
	\end{tikzpicture}
	\ &=\
    \begin{tikzpicture}[anchorbase]
	    \draw[->] (-1.2,0) to (-1.2,1);
	    \node at (-0.8,0.5) {$\cdots$};
	    \node at (-0.8,0) {\regionlabel{m}};
	    \draw[->] (-0.4,0) to (-0.4,1);
	   \draw[->] (0.8,0) to  (0,1);
        \draw[wipe] (0.4,0) to[out=up, in=down] (0.8,0.5) to[out=up,in=down] (0.4,1);
        \draw[->] (0.4,0) to[out=up, in=down] (0.8,0.5) to[out=up,in=down] (0.4,1);
        \draw[wipe] (0,0) to  (0.8,1);
        \draw[->] (0,0) to  (0.8,1);	
        \draw[->] (1.2,0) to (1.2,1);
	    \node at (1.6,0.5) {$\cdots$};
	    \node at (1.6,0) {\regionlabel{n}};
	    \draw[->] (2,0) to (2,1);
	\end{tikzpicture},
\end{split}
\end{align}
for all $m,n \in \N$, and $C_\cB$ consists of morphisms
\begin{align} \label{Briad-Int}
    \begin{tikzpicture}[anchorbase]
	    \draw[->] (-1.2,0) to (-1.2,1);
	    \node at (-0.8,0.5) {$\cdots$};
	    \node at (-0.8,0) {\regionlabel{m}};
	    \draw[->] (-0.4,0) to (-0.4,1);
	    \draw[->] (0,0) to (0,0.5) to [out=90, in=270]  (0.4,1);
	    \draw[wipe] (0.4,0) to (0.4,0.5) to [out=90, in=270]  (0,1);
	    \draw[->] (0.4,0) to (0.4,0.5) to [out=90, in=270]  (0,1);
	    \draw[->] (0.8,0) to (0.8,1);
	    \node at (1.2,0.5) {$\cdots$};
	    \node at (1.2,0) {\regionlabel{l}};
	    \draw[->] (1.6,0) to (1.6,1);
	    \draw[->] (2,0) to [out=90, in=270]  (2.4,0.5) to (2.4,1);
	    \draw[wipe] (2.4,0) to [out=90, in=270]  (2,0.5) to (2,1);
	    \draw[->] (2.4,0) to [out=90, in=270]  (2,0.5) to (2,1);
	    \draw[->] (2.8,0) to (2.8,1);
	    \node at (3.2,0.5) {$\cdots$};
	    \node at (3.2,0) {\regionlabel{n}};
	    \draw[->] (3.6,0) to (3.6,1);
	\end{tikzpicture}
	\ =\
	    \begin{tikzpicture}[anchorbase]
	    \draw[->] (-1.2,0) to (-1.2,1);
	    \node at (-0.8,0.5) {$\cdots$};
	    \node at (-0.8,0) {\regionlabel{m}};
	    \draw[->] (-0.4,0) to (-0.4,1);
	    \draw[->] (0,0) to [out=90, in=270] (0.4,0.5) to   (0.4,1);
	    \draw[wipe] (0.4,0) to  [out=90, in=270]  (0,0.5) to (0,1);
	    \draw[->] (0.4,0) to  [out=90, in=270]  (0,0.5) to (0,1);
	    \draw[->] (0.8,0) to (0.8,1);
	    \node at (1.2,0.5) {$\cdots$};
	    \node at (1.2,0) {\regionlabel{l}};
	    \draw[->] (1.6,0) to (1.6,1);
	    \draw[->] (2,0) to  (2.0,0.5) to  [out=90, in=270] (2.4,1);
	    \draw[wipe] (2.4,0) to  (2.4 ,0.5) to [out=90, in=270] (2,1);
	    \draw[->] (2.4,0) to  (2.4 ,0.5) to [out=90, in=270] (2,1);
	    \draw[->] (2.8,0) to (2.8,1);
	    \node at (3.2,0.5) {$\cdots$};
	    \node at (3.2,0) {\regionlabel{n}};
	    \draw[->] (3.6,0) to (3.6,1);
	\end{tikzpicture},
\end{align}
where $m,n,l\in \N$.
\end{itemize}

Fix a natural number $d$, we have the endomoprhism algebra $\End(\uparrow^d)$ is generated by
\begin{equation}\label{braid-end-generator}
    s_i=
    \begin{tikzpicture}[anchorbase]
	    \draw[->] (-1.2,0) to (-1.2,0.5);
	    \node at (-0.8,0.25) {$\cdots$};
	    \node[below] at (-0.8,0) {\regionlabel{d-1-i}};
	    \draw[->] (-0.4,0) to (-0.4,0.5);
	    \draw[->] (0.4,0) to [out=90, in=270]  (0,0.5);
	    \draw[wipe] (0,0) to [out=90, in=270]  (0.4,0.5);
	    \draw[->] (0,0) to [out=90, in=270]  (0.4,0.5);
	    \draw[->] (0.8,0) to (0.8,0.5);
	    \node at (1.2,0.25) {$\cdots$};
	    \node[below] at (1.2,0) {\regionlabel{i-1}};
	    \draw[->] (1.6,0) to (1.6,0.5);
	\end{tikzpicture},
		\quad \text{and} \quad
		s_i^{-1}=
	    \begin{tikzpicture}[anchorbase]
	    \draw[->] (-1.2,0) to (-1.2,0.5);
	    \node at (-0.8,0.25) {$\cdots$};
	    \node[below] at (-0.8,0) {\regionlabel{d-1-i}};
	    \draw[->] (-0.4,0) to (-0.4,0.5);
	    \draw[->] (0,0) to [out=90, in=270]  (0.4,0.5);
	    \draw[wipe] (0.4,0) to [out=90, in=270]  (0,0.5);
	    \draw[->] (0.4,0) to [out=90, in=270]  (0,0.5);
	    \draw[->] (0.8,0) to (0.8,0.5);
	    \node at (1.2,0.25) {$\cdots$};
	    \node[below] at (1.2,0) {\regionlabel{i-1}};
	    \draw[->] (1.6,0) to (1.6,0.5);
	\end{tikzpicture},
\end{equation}
where $i=1,\dots,d-1$, and they satisfy the relation
\begin{align}\label{braid-end-rel}
\begin{split}
    s_{i+1} s_i s_{i+1} &= s_i s_{i+1} s_i,  \quad i=1,\dots,d_1\\
    s_is_j &= s_js_i, \qquad \left|i-j\right|>1.
\end{split}
\end{align}
This presentation is the presentation of the group algebra of the briad group on $d$ strands.

\subsection{Hecke algebras}

Fix $z \in \kk^\times$.  Let $\cH(z)$ be the strict $\kk$-linear monoidal category $\cB$ defined in \cref{braid-cat}, but with one more relation:
\begin{equation} \label{skein}
  \begin{tikzpicture}[anchorbase]
    \draw[->] (0.25,-0.25) to (-0.25,0.25);
    \draw[wipe] (-0.25,-0.25) to (0.25,0.25);
    \draw[->] (-0.25,-0.25) to (0.25,0.25);
  \end{tikzpicture}
  \ -\
  \begin{tikzpicture}[anchorbase]
    \draw[->] (-0.25,-0.25) to (0.25,0.25);
    \draw[wipe] (0.25,-0.25) to (-0.25,0.25);
    \draw[->] (0.25,-0.25) to (-0.25,0.25);
  \end{tikzpicture}
  \ = z\
  \begin{tikzpicture}[anchorbase]
    \draw[->] (0.2,-0.25) to (0.2,0.25);
    \draw[->] (-0.2,-0.25) to (-0.2,0.25);
  \end{tikzpicture}
  \ .
\end{equation}
We take $\kk = \C(q)$ and $z = q - q^{-1}$ and fix $d\in \N$. Since all morphisms in $\cH$ are endomorphisms, \cref{thm:presentationofalgebra} applies. We have that all generators in $\End_{\cH(z)}(\uparrow^d)$ are those in \cref{braid-end-generator} and all relations are those in \cref{braid-end-rel} together with
\begin{align}
    \begin{tikzpicture}[anchorbase, yscale=1.2]
	    \draw[->] (-1.2,0) to (-1.2,0.5);
	    \node at (-0.8,0.25) {$\cdots$};
	    \node[below] at (-0.8,0) {\regionlabel{d-1-i}};
	    \draw[->] (-0.4,0) to (-0.4,0.5);
	    \draw[->] (0.4,0) to  [out=90, in=270]  (0,0.5);
	    \draw[wipe] (0,0) to [out=90, in=270] (0.4,0.5);
	    \draw[->] (0,0) to [out=90, in=270] (0.4,0.5);
	    \draw[->] (0.8,0) to (0.8,0.5);
	    \node at (1.2,0.25) {$\cdots$};
	    \node[below] at (1.2,0) {\regionlabel{i-1}};
	    \draw[->] (1.6,0) to (1.6,0.5);
    \end{tikzpicture}
    \ -\
    \begin{tikzpicture}[anchorbase, yscale=1.2]
	    \draw[->] (-1.2,0) to (-1.2,0.5);
	    \node at (-0.8,0.25) {$\cdots$};
	    \node[below] at (-0.8,0) {\regionlabel{d-1-i}};
	    \draw[->] (-0.4,0) to (-0.4,0.5);
	    \draw[->] (0,0)to [out=90, in=270]  (0.4,0.5);
	    \draw[wipe] (0.4,0)  to [out=90, in=270]  (0,0.5);
	    \draw[->] (0.4,0)  to [out=90, in=270]  (0,0.5);
	    \draw[->] (0.8,0) to (0.8,0.5);
	    \node at (1.2,0.25) {$\cdots$};
	    \node[below] at (1.2,0) {\regionlabel{i-1}};
	    \draw[->] (1.6,0) to (1.6,0.5);
	\end{tikzpicture}
	\ =\
	(q - q^{-1})
	 \begin{tikzpicture}[anchorbase, yscale=1.2]
	    \draw[->] (-1.2,0) to (-1.2,0.5);
	    \node at (-0.8,0.25) {$\cdots$};
	    \node[below] at (-0.8,0) {\regionlabel{d-1-i}};
	    \draw[->] (-0.4,0) to (-0.4,0.5);
	    \draw[->] (0,0) to (0,0.5);
	    \draw[->] (0.4,0) to (0.4,0.5);
	    \draw[->] (0.8,0) to (0.8,0.5);
	    \node at (1.2,0.25) {$\cdots$};
	    \node[below] at (1.2,0) {\regionlabel{i-1}};
	    \draw[->] (1.6,0) to (1.6,0.5);
	\end{tikzpicture}.
\end{align}
then $\End_{\cH(z)} (\uparrow^{\otimes d})$ is the \emph{Iwahori--Hecke algebra} of type $A_{d-1}$.


\subsection{Wreath product algebras}
In this section we give the definition of wreath product categories by following the presentation in  \cite[\S 3.6]{Savage18}.
Let $A$ be an associative $\kk$-algebra. Define the \emph{wreath product category} $\cW(A)$ to be the strict $\kk$-linear monoidal category obtained from $\cS$ by adding morphisms such that we have an algebra homomorphipsm
\[
  A \to \End(\uparrow),
  \qquad
  a \mapsto
  \begin{tikzpicture}[anchorbase]
    \draw[->] (0,0) to (0,0.6);
    \bluedot{(0,0.3)} node[anchor=west,color=black] {\dotlabel{a}};
  \end{tikzpicture}
  \ .
\]
In particular, this means that
\begin{equation} \label{dotlin}
  \begin{tikzpicture}[anchorbase]
    \draw[->] (0,0) to (0,0.6);
    \bluedot{(0,0.3)} node[anchor=west,color=black] {\dotlabel{(\alpha a+ \alpha b)}};
  \end{tikzpicture}
  = \alpha\
  \begin{tikzpicture}[anchorbase]
    \draw[->] (0,0) to (0,0.6);
    \bluedot{(0,0.3)} node[anchor=west,color=black] {\dotlabel{a}};
  \end{tikzpicture}
  + \beta\
  \begin{tikzpicture}[anchorbase]
    \draw[->] (0,0) to (0,0.6);
    \bluedot{(0,0.3)} node[anchor=west,color=black] {\dotlabel{b}};
  \end{tikzpicture}
  \qquad \text{and} \qquad
  \begin{tikzpicture}[anchorbase]
    \draw[->] (0,0) to (0,1);
    \bluedot{(0,0.3)} node[anchor=west,color=black] {\dotlabel{b}};
    \bluedot{(0,0.6)} node[anchor=west,color=black] {\dotlabel{a}};
  \end{tikzpicture}
  =\
  \begin{tikzpicture}[anchorbase]
    \draw[->] (0,0) to (0,1);
    \bluedot{(0,0.5)} node[anchor=west,color=black] {\dotlabel{ab}};
  \end{tikzpicture}
  \qquad \text{for all } \alpha,\beta \in \kk,\ a,b \in A.
\end{equation}
We call the closed circles appearing in the above diagrams \emph{tokens}.  We then impose the additional relation
\begin{equation} \label{tokslide}
  \begin{tikzpicture}[anchorbase, scale=0.75]
    \draw[->] (0,0) -- (1,1);
    \draw[->] (1,0) -- (0,1);
    \bluedot{(.25,.25)} node [anchor=east, color=black] {\dotlabel{a}};
  \end{tikzpicture}
  \ =\
  \begin{tikzpicture}[anchorbase, scale=0.75]
    \draw[->](0,0) -- (1,1);
    \draw[->](1,0) -- (0,1);
    \bluedot{(0.75,.75)} node [anchor=east, color=black] {\dotlabel{a}};
  \end{tikzpicture}
  \ ,\quad a \in A.
\end{equation}
Note that we can compose \cref{tokslide} on the top and bottom with a crossing to obtain
\[
  \begin{tikzpicture}[anchorbase,scale=0.75]
    \draw[->] (0,0) -- (1,1);
    \draw[->] (1,0) -- (0,1);
    \bluedot{(.25,.25)} node [anchor=east, color=black] {\dotlabel{a}};
  \end{tikzpicture}
  \ =\
  \begin{tikzpicture}[anchorbase,scale=0.75]
    \draw[->](0,0) -- (1,1);
    \draw[->](1,0) -- (0,1);
    \bluedot{(0.75,.75)} node [anchor=east, color=black] {\dotlabel{a}};
  \end{tikzpicture}
  \ \implies\
  \begin{tikzpicture}[anchorbase]
    \draw[->] (0.25,-0.5) to[out=up,in=down] (-0.25,0) to[out=up,in=down] (0.25,0.5) to[out=up,in=down] (-0.25,1);
    \draw[->] (-0.25,-0.5) to[out=up,in=down] (0.25,0) to[out=up,in=down] (-0.25,0.5) to[out=up,in=down] (0.25,1);
    \bluedot{(-0.25,0)} node[anchor=east,color=black] {\dotlabel{a}};
  \end{tikzpicture}
  \ =\
  \begin{tikzpicture}[anchorbase]
    \draw[->] (0.25,-0.5) to[out=up,in=down] (-0.25,0) to[out=up,in=down] (0.25,0.5) to[out=up,in=down] (-0.25,1);
    \draw[->] (-0.25,-0.5) to[out=up,in=down] (0.25,0) to[out=up,in=down] (-0.25,0.5) to[out=up,in=down] (0.25,1);
    \bluedot{(0.25,0.5)} node[anchor=west,color=black] {\dotlabel{a}};
  \end{tikzpicture}
  \ \stackrel{\cref{Sn-strings}}{\implies} \
  \begin{tikzpicture}[anchorbase,scale=0.75]
    \draw[->] (0,0) -- (1,1);
    \draw[->] (1,0) -- (0,1);
    \bluedot{(0.25,.75)} node [anchor=north east, color=black] {\dotlabel{a}};
  \end{tikzpicture}
  \ =\
  \begin{tikzpicture}[anchorbase,scale=0.75]
    \draw[->] (0,0) -- (1,1);
    \draw[->] (1,0) -- (0,1);
    \bluedot{(.75,.25)} node [anchor=south west, color=black] {\dotlabel{a}};
  \end{tikzpicture}
  \ .
\]
So tokens also slide up-left through crossings.

We have the following theorem in \cite[\S  3.6]{Savage18} stated without proof.
\begin{prop}
The endomoprhism algebras of the wreath product algebra satisfy
\[
  \End_{\cW(A)} (\uparrow^{\otimes d}) \cong A^{\otimes d} \rtimes S_d,
\]
which is called the $d$-th \emph{wreath product algebra} associated to $A$, where the multiplication is determined by
\[
  (\baone \otimes \pi_1) (\batwo \otimes \pi_2)
  = (\baone (\pi_1\cdot \ba_{2})\otimes \pi_1 \pi_2),
  \quad \baone, \batwo \in A^{\otimes d},\ \pi_1, \pi_2 \in S_d,
\]
where $\pi_1 \cdot a_2$ denotes the natural action of $\pi_1 \in S_d$ on $a_2 \in A^{\otimes d}$ by permutation of the factors.
\end{prop}
\begin{proof}
All generators of  $\End_{\cW(A)} (\uparrow^{\otimes d})$ are endomorphisms; thus \cref{thm:presentationofalgebra} applies. The generators are $s_i$ defined in \cref{sym-end-gen} together with
\[
u_j(a) =
 \begin{tikzpicture}[anchorbase]
	    \draw[->] (-1.2,0) to (-1.2,0.5);
	    \node at (-0.8,0.25) {$\cdots$};
	    \node at (-0.8,0) {\regionlabel{d-j}};
	    \draw[->] (-0.4,0) to (-0.4,0.5);
	    \draw[->] (0,0) to (0,0.5);
	    \bluedot{(0,0.25)} node[anchor=west,color=black] {\dotlabel{a}};
	    \draw[->] (0.4,0) to (0.4,0.5);
	    \node at (0.8,0.25) {$\cdots$};
	    \node at (0.8,0) {\regionlabel{j-1}};
	    \draw[->] (1.2,0) to (1.2,0.5);
	\end{tikzpicture},
	\quad a \in A.
\]
By \cref{thm:presentationofalgebra}, the generators $s_i$ satisfy the relations of $\cS$ as in \cref{sym-end-rel}, and the generators $u_j(a)$ satisfy the relations
\begin{align}\label{prod-end-rel}
\begin{split}
  \begin{tikzpicture}[anchorbase]
    \draw[->] (-1.2,0) to (-1.2,0.5);
    \node at (-0.8,0.25) {$\cdots$};
    \node at (-0.8,0) {\regionlabel{d-j}};
    \draw[->] (-0.4,0) to (-0.4,0.5);
    \draw[->] (0,0) to (0,0.5);
    \bluedot{(0,0.25)} node[anchor=west,color=black] {\dotlabel{(\alpha a+ \beta b)}};
    \draw[->] (1.6,0) to (1.6,0.5);
    \node at (2,0.25) {$\cdots$};
    \node at (2,0) {\regionlabel{j-1}};
    \draw[->] (2.4,0) to (2.4,0.5);
  \end{tikzpicture}
  =& \alpha\
    \begin{tikzpicture}[anchorbase]
    \draw[->] (-1.2,0) to (-1.2,0.5);
    \node at (-0.8,0.25) {$\cdots$};
    \node at (-0.8,0) {\regionlabel{d-j}};
    \draw[->] (-0.4,0) to (-0.4,0.5);
    \draw[->] (0,0) to (0,0.5);
    \bluedot{(0,0.25)} node[anchor=west,color=black] {\dotlabel{a}};
    \draw[->] (0.4,0) to (0.4,0.5);
    \node at (0.8,0.25) {$\cdots$};
    \node at (0.8,0) {\regionlabel{j-1}};
    \draw[->] (1.2,0) to (1.2,0.5);
  \end{tikzpicture}
  + \beta\
    \begin{tikzpicture}[anchorbase]
    \draw[->] (-1.2,0) to (-1.2,0.5);
    \node at (-0.8,0.25) {$\cdots$};
    \node at (-0.8,0) {\regionlabel{d-j}};
    \draw[->] (-0.4,0) to (-0.4,0.5);
    \draw[->] (0,0) to (0,0.5);
    \bluedot{(0,0.25)} node[anchor=west,color=black] {\dotlabel{b}};
    \draw[->] (0.4,0) to (0.4,0.5);
    \node at (0.8,0.25) {$\cdots$};
    \node at (0.8,0) {\regionlabel{j-1}};
    \draw[->] (1.2,0) to (1.2,0.5);
  \end{tikzpicture},\\
    \begin{tikzpicture}[anchorbase]
    \draw[->] (-1.2,0) to (-1.2,1);
    \node at (-0.8,0.5) {$\cdots$};
    \node at (-0.8,0) {\regionlabel{d-j}};
    \draw[->] (-0.4,0) to (-0.4,1);
    \draw[->] (0,0) to (0,1);
    \bluedot{(0,0.75)} node[anchor=west,color=black] {\dotlabel{a}};
    \bluedot{(0,0.25)} node[anchor=west,color=black] {\dotlabel{b}};
    \draw[->] (0.4,0) to (0.4,1);
    \node at (0.8,0.5) {$\cdots$};
    \node at (0.8,0) {\regionlabel{j-1}};
    \draw[->] (1.2,0) to (1.2,1);
  \end{tikzpicture}
  =&\
   \begin{tikzpicture}[anchorbase]
    \draw[->] (-1.2,0) to (-1.2,0.5);
    \node at (-0.8,0.25) {$\cdots$};
    \node at (-0.8,0) {\regionlabel{d-j}};
    \draw[->] (-0.4,0) to (-0.4,0.5);
    \draw[->] (0,0) to (0,0.5);
    \bluedot{(0,0.25)} node[anchor=west,color=black] {\dotlabel{ab}};
    \draw[->] (0.6,0) to (0.6,0.5);
    \node at (1,0.25) {$\cdots$};
    \node at (1,0) {\regionlabel{j-1}};
    \draw[->] (1.4,0) to (1.4,0.5);
  \end{tikzpicture},\\
    \begin{tikzpicture}[anchorbase]
	    \draw[->] (-1.2,0) to (-1.2,1);
	    \node at (-0.8,0.5) {$\cdots$};
	    \node[below] at (-0.8,0) {\regionlabel{d-j}};
	    \draw[->] (-0.4,0) to (-0.4,1);
	    \draw[->] (0,0) to (0,1);
	    \bluedot{(0,0.75)} node[anchor=west,color=black] {\dotlabel{b}};
	    \draw[->] (0.4,0) to (0.4,1);
	    \node at (0.8,0.5) {$\cdots$};
	    \node[below] at (0.8,0) {\regionlabel{(j-i)-1}};
	    \draw[->] (1.2,0) to (1.2,1);
	  	\draw[->] (1.6,0) to (1.6,1);
	     \bluedot{(1.6,0.25)} node[anchor=west,color=black] {\dotlabel{a}};
	    \draw[->] (2,0) to (2, 1);
	    \node at (2.4,0.5) {$\cdots$};
	    \node[below] at (2.4,0) {\regionlabel{i-1}};
	    \draw[->] (2.8,0) to (2.8,1);
	\end{tikzpicture}
	\ =&\
	 \begin{tikzpicture}[anchorbase]
	    \draw[->] (-1.2,0) to (-1.2,1);
	    \node at (-0.8,0.5) {$\cdots$};
	    \node[below] at (-0.8,0) {\regionlabel{d-j}};
	    \draw[->] (-0.4,0) to (-0.4,1);
	    \draw[->] (0,0) to (0,1);
	    \bluedot{(0,0.25)} node[anchor=west,color=black] {\dotlabel{b}};
	    \draw[->] (0.4,0) to (0.4,1);
	    \node at (0.8,0.5) {$\cdots$};
	    \node[below] at (0.8,0) {\regionlabel{(j-i)-1}};
	    \draw[->] (1.2,0) to (1.2,1);
	  	\draw[->] (1.6,0) to (1.6,1);
	    \bluedot{(1.6,0.75)} node[anchor=west,color=black] {\dotlabel{a}};
	    \draw[->] (2,0) to (2, 1);
	    \node at (2.4,0.5) {$\cdots$};
	    \node[below] at (2.4,0) {\regionlabel{i-1}};
	    \draw[->] (2.8,0) to (2.8,1);
	\end{tikzpicture},
\end{split}
\end{align}
where $\alpha,\beta \in \kk$ and $a,b \in A$.
Most importantly, by \cref{tokslide}, these two kinds of generators together satisfy the following relations:
\begin{align} \label{prod-end-rel-int}
\begin{split}
    \begin{tikzpicture}[anchorbase]
	    \draw[->] (-1.2,0) to (-1.2,1);
	    \node at (-0.8,0.5) {$\cdots$};
	    \node[below] at (-0.8,0) {\regionlabel{d-1-i}};
	    \draw[->] (-0.4,0) to (-0.4,1);
	    \draw[->] (0,0) to (0,0.5) to [out=90, in=270]  (0.4,1);
	    \draw[->] (0.4,0) to (0.4,0.5) to [out=90, in=270]  (0,1);
	    \bluedot{(0.4,0.25)} node[anchor=west,color=black] {\dotlabel{a}};
	    \draw[->] (0.8,0) to (0.8,1);
	    \node at (1.2,0.5) {$\cdots$};
	    \node[below] at (1.2,0) {\regionlabel{i-1}};
	    \draw[->] (1.6,0) to (1.6,1);
	\end{tikzpicture}
    \ =&\
    \begin{tikzpicture}[anchorbase]
	    \draw[->] (-1.2,0) to (-1.2,1);
	    \node at (-0.8,0.5) {$\cdots$};
	    \node[below] at (-0.8,0) {\regionlabel{d-1-i}};
	    \draw[->] (-0.4,0) to (-0.4,1);
	    \draw[->] (0,0) to [out=90, in=270] (0.4,0.5) to   (0.4,1);
	    \draw[->] (0.4,0) to  [out=90, in=270]  (0,0.5) to (0,1);
	    {(0,0.75)};
	    \bluedot{(0,0.75)} node[anchor=west,color=black] {\dotlabel{a}};
	    \draw[->] (0.8,0) to (0.8,1);
	    \node at (1.2,0.5) {$\cdots$};
	    \node[below] at (1.2,0) {\regionlabel{i-1}};
	    \draw[->] (1.6,0) to (1.6,1);
    \end{tikzpicture},
	\\
	\begin{tikzpicture}[anchorbase]
	    \draw[->] (-1.6,0) to (-1.6,1);
	    \node at (-1.2,0.5) {$\cdots$};
	    \node[below] at (-1.2,0) {\regionlabel{d-1-j}};
	    \draw[->] (-0.8,0) to (-0.8,1);
        \draw[->] (-0.4,0) to  (-0.4,0.5) to [out=90, in=270] (0,1);
	    \draw[->] (0,0) to (0,0.5) to  [out=90, in=270]  (-0.4,1);
        \draw[->] (0.4,0) to (0.4,1);
	    \node at (0.8,0.5) {$\cdots$};
	    \node[below] at (0.8,0) {\regionlabel{(j-i)-1}};
	    \draw[->] (1.2,0) to (1.2,1);
	  	\draw[->] (1.6,0) to (1.6,1);
	    \bluedot{(1.6,0.25)} node[anchor=west,color=black] {\dotlabel{a}};
	    \draw[->] (2,0) to (2, 1);
	    \node at (2.4,0.5) {$\cdots$};
	    \node[below] at (2.4,0) {\regionlabel{i-1}};
	    \draw[->] (2.8,0) to (2.8,1);
	    \end{tikzpicture}
	\ =&\
	\begin{tikzpicture}[anchorbase]
	    \draw[->] (-1.6,0) to (-1.6,1);
	    \node at (-1.2,0.5) {$\cdots$};
	    \node[below] at (-1.2,0) {\regionlabel{d-1-j}};
	    \draw[->] (-0.8,0) to (-0.8,1);
        \draw[->] (-0.4,0) to [out=90, in=270] (0,0.5) to  (0,1);
	    \draw[->] (0,0) to  [out=90, in=270]  (-0.4,0.5) to (-0.4,1);
        \draw[->] (0.4,0) to (0.4,1);
	    \node at (0.8,0.5) {$\cdots$};
	    \node[below] at (0.8,0) {\regionlabel{(j-i)-1}};
	    \draw[->] (1.2,0) to (1.2,1);
	  	\draw[->] (1.6,0) to (1.6,1);
	   \bluedot{(1.6,0.75)} node[anchor=west,color=black] {\dotlabel{a}};
	    \draw[->] (2,0) to (2, 1);
	    \node at (2.4,0.5) {$\cdots$};
	    \node[below] at (2.4,0) {\regionlabel{i-1}};
	    \draw[->] (2.8,0) to (2.8,1);
	\end{tikzpicture},
\end{split}
\end{align}
where $a \in A$. Algebraically, \cref{prod-end-rel-int} reads
\[
    s_i u_j(b) = u_{s_i(j)}(b) s_i = (s_i \cdot u_j)(b) s_i, \quad \text{for all } b \in A,
\]
where we identify the generator $s_i$ with the transiposition $(i,i+1)$ and thus $s_i(j)$ denote the image of $j$ under $(i,i+1).$

Let $B$ be the free algebra generated by $u_j(b),\ j=1,\dots,n,\ b \in A$ and $s_i,\ i=1,\dots ,d-1$, we define a $\kk$-algebra morphism $\phi$ from $B$ to $A^{\otimes d} \rtimes S_d$ determined by the following:
\begin{align*}
    \phi(s_i) =& 1^{\otimes d} \otimes s_i, \\
    \phi(u_j(b))=& 1^{\otimes {d-j}}
    \otimes {b} \otimes 1^{\otimes j} \otimes \id_d,
\end{align*}
where $s_i$ denotes the transposition $(i,i+1)$ and $\id_d$ is the identity of $S_d$.
It is clear that $\phi$ is surjective and the kernel of $\phi$ contains all the relations on the generators of $\End_{\cW(A)}\left(\uparrow^d\right)$. In particular, \cref{prod-end-rel-int} is satisfied due to the product
\begin{align}
\begin{split}
    (\baone \otimes \pi_1)(\batwo \otimes \pi_2)
    =& (\baone(\pi_1 \cdot \batwo) \otimes \pi_1\pi_2).
\end{split}
\end{align}
Thus $\phi$ induces a surjective algebra morphism $\bar{\phi}$ from $\End_{\cW(A)}\left(\uparrow^d\right)$ to  $A^{\otimes d} \rtimes S_d$. Conversely, consider the following multi-linear map
\begin{align*}
      \psi: A^{d}\times S_d
      \ \to& \
     \End_{\cW(A)}(\uparrow^{\otimes d}),  \\
    (a_d\times \dotsb\times a_1 \times \pi) \ \mapsto&\ u_d(a_d)\dotsm u_1(a_1) \pi .
\end{align*}
It induces a natural $\kk$-module morphism as follows:
\begin{align*}
      \bar{\psi}: A^{\otimes d} \rtimes S_d= A^{\otimes d} \otimes \kk S_d
      \ \to& \
      \End_{\cW(A)}(\uparrow^{\otimes d}),  \\
    (a_d\otimes \dotsb\otimes a_1)\otimes \pi \ \mapsto&\ u_d(a_d)\dotsm u_1(a_1) \pi .
\end{align*}
The $\kk$-module morphism $\bar{\psi}$ actually forms an algebra morphism: for $\baone=(a_i)_{i=1}^d,\batwo=(b_i)_{i=1}^d \in A^{\otimes d}$ and $\pi_1,\pi_2 \in S_d$,
\begin{align}
\begin{split}
    \bar{\psi}((\baone \otimes \pi_1)(\batwo \otimes \pi_2))
    =& \bar{\psi} (\baone(\pi_1 \cdot \batwo) \otimes \pi_1\pi_2)\\
    =& u_d(a_d)\dotsm u_1(a_1)u_d(b_{\pi_1(d)}) ...u_1(b_{\pi_1(1)})\pi_1\pi_2\\
    \stackrel{\star}{=}& u_d(a_d)\dotsm u_1(a_1)\pi_1u_d(b_d) ...u_1(b_1)\pi_2\\
    =& \bar{\psi}(\baone \otimes \pi_1)\bar{\psi}(\batwo \otimes \pi_2),
\end{split}
\end{align}
where one can check that $\star$ holds by writting $\pi_1$ as a product of transpositions.

Since $\bar{\psi}$ is $\kk$-algebra morphism, it is clear that it is the inverse of $\bar{\phi}$.
This shows that $\End_\cW(A)\left(\uparrow^d\right)$ is isomorphic to $A^{\otimes d} \rtimes  S_d$.
\end{proof}

\subsection{Affine wreath product algebras}

The wreath product category $\cW(A)$ is a generalization of the symmetric group category $\cS$ that depends on a choice of associative $\kk$-algebra $A$. We now suppose that we have a $\kk$-linear \emph{trace map}
$  \tr \colon A \to \kk $
and dual bases $B$ and $\left\lbrace \chk{b} : b \in B\right\rbrace  $ of $A$ such that
\[
  \tr(ab) = \tr(ba) \quad \text{and} \quad
  \tr(\chk{a}b) = \delta_{a,b}
  \quad \text{for all } a,b \in B.
\]
It can be shown that
$  \sum_{b \in B} b \otimes \chk{b} \in A \otimes A$
is independent of the choice of basis $B$. Also, for all $x \in A$, we have
\begin{equation} \label{teleport}
  \begin{split}
    \sum_{b \in B} bx \otimes \chk{b}
    = \sum_{a \in B} a \otimes x\chk{a}.
  \end{split}
\end{equation}

We define the \emph{affine wreath product category} $\AW(A)$ to be the strict $\kk$-linear monoidal category obtained from $\cW(A)$ by adding a generating morphism
$
  \begin{tikzpicture}[anchorbase]
    \draw[->] (0,0) to (0,0.6);
    \redcircle{(0,0.3)};
  \end{tikzpicture}
  \ \colon \uparrow\ \to\ \uparrow
$
and the additional relations
\begin{equation} \label{AWPA}
  \begin{tikzpicture}[anchorbase]
    \draw[->] (0,0) -- (0.6,0.6);
    \draw[->] (0.6,0) -- (0,0.6);
    \redcircle{(0.15,.45)};
  \end{tikzpicture}
  \ -\
  \begin{tikzpicture}[anchorbase]
    \draw[->] (0,0) -- (0.6,0.6);
    \draw[->] (0.6,0) -- (0,0.6);
    \redcircle{(.45,.15)};
  \end{tikzpicture}
  \ = \sum_{b \in B}
  \begin{tikzpicture}[anchorbase]
    \draw[->] (0,0) -- (0,0.6);
    \draw[->] (0.3,0) -- (0.3,0.6);
    \bluedot{(0,0.3)} node[anchor=east,color=black] {\dotlabel{b}};
    \bluedot{(0.3,0.3)} node[anchor=west,color=black] {\dotlabel{\chk{b}}};
  \end{tikzpicture}
  \qquad \text{and} \qquad
  \begin{tikzpicture}[anchorbase]
    \draw[->] (0,0) to (0,1);
    \redcircle{(0,0.3)};
    \bluedot{(0,0.6)} node[anchor=east,color=black] {\dotlabel{a}};
  \end{tikzpicture}
  \ =
  \begin{tikzpicture}[anchorbase]
    \draw[->] (0,0) to (0,1);
    \redcircle{(0,0.6)};
    \bluedot{(0,0.3)} node[anchor=west,color=black] {\dotlabel{a}};
  \end{tikzpicture}
  \ ,\quad \text{for all } a \in A.
\end{equation}

\begin{prop} The endomorphism algebra
$  \End_{\AW(A)}(\uparrow^{\otimes d}) $ in the affine wreath product category is an affine wreath product algebra as defined in \cite[Definition 3.1]{Sav17}. Namely, as an $\kk$-algebra, $  \End_{\AW(A)}(\uparrow^{\otimes d}) $  is generated by
\[
t_i, \ i=1,\dots,d,   \quad u_i(a),\  i=1,\dots,d,\  a\in A, \quad \text{and} \quad s_i,\ i=1,\dots,d-1,
\]
where they satisfy  \cref{sym-end-rel,ahdeg-end-interchange,prod-end-rel,prod-end-rel-int} and the following relations:
\begin{align}
\begin{split}
    t_iu_j(a) = u_j(a)t_i, \quad &\text{for all } a \in A\\
    t_{i+1}s_i - s_it_i = m_i, \quad &\text{for }i=1,\dots,d-1,
\end{split}
\end{align}
where $m_i=\sum_{b \in B} u_{i+1}(b)u_i(\chk{b})$.
\end{prop}
\begin{proof}
As a linear monoidal category, all generators of $\AW(A)$ are endormophisms, thus \cref{thm:presentationofalgebra} applies. Thus for fixed $d$, we have that the generators of
$ \End_{\AW(A)}(\uparrow^{\otimes d}) $
are the generators of $\End_{\cW(A)}(\uparrow^{\otimes d})$ together with
\begin{equation}
    t_i =
	  \begin{tikzpicture}[anchorbase]
	    \draw[->] (-1.2,0) to (-1.2,0.5);
	    \node at (-0.8,0.25) {$\cdots$};
	    \node[below] at (-0.8,0) {\regionlabel{d-j}};
	    \draw[->] (-0.4,0) to (-0.4,0.5);
	    \draw[->] (0,0) to (0,0.5);
	    \redcircle{(0,0.25)};
	    \draw[->] (0.4,0) to (0.4,0.5);
	    \node at (0.8,0.25) {$\cdots$};
	    \node[below] at (0.8,0) {\regionlabel{j-1}};
	    \draw[->] (1.2,0) to (1.2,0.5);
	\end{tikzpicture}.
\end{equation}
The relations they satisfy in $\AW(A)$ should be the relations they satisfy in $\cW(A)$, namely \cref{sym-end-rel,prod-end-rel,prod-end-rel-int}, together with the interchange law involving $t_i$ and the new relation induced by the first relation in \cref{AWPA}. To be precise, the interchange law involving $t_i$ corresponds to \cref{ahdeg-end-interchange} and
\[
 \begin{tikzpicture}[anchorbase]
	    \draw[->] (-1.2,0) to (-1.2,1);
	    \node at (-0.8,0.5) {$\cdots$};
	    \node[below] at (-0.8,0) {\regionlabel{d-j}};
	    \draw[->] (-0.4,0) to (-0.4,1);
	    \draw[->] (0,0) to (0,1);
	    \redcircle{(0,0.25)};
	    \bluedot{(0,0.75)} node[anchor=west, color=black] {\dotlabel{a}};
	    \draw[->] (0.4,0) to (0.4,1);
	    \node at (0.8,0.5) {$\cdots$};
	    \node[below] at (0.8,0) {\regionlabel{j-1}};
	    \draw[->] (1.2,0) to (1.2,1);
	\end{tikzpicture}
	=
	 \begin{tikzpicture}[anchorbase]
	    \draw[->] (-1.2,0) to (-1.2,1);
	    \node at (-0.8,0.5) {$\cdots$};
	    \node[below] at (-0.8,0) {\regionlabel{d-j}};
	    \draw[->] (-0.4,0) to (-0.4,1);
	    \draw[->] (0,0) to (0,1);
	    \redcircle{(0,0.75)};
	    \bluedot{(0,0.25)} node[anchor=west, color=black] {\dotlabel{a}};
	    \draw[->] (0.4,0) to (0.4,1);
	    \node at (0.8,0.5) {$\cdots$};
	    \node[below] at (0.8,0) {\regionlabel{j-1}};
	    \draw[->] (1.2,0) to (1.2,1);
	\end{tikzpicture}.
\]
The induced relations, by \cref{MainTheorem}, are
\[
    \begin{tikzpicture}[anchorbase]
	    \draw[->] (-1.2,0) to (-1.2,1);
	    \node at (-0.8,0.5) {$\cdots$};
	    \node[below] at (-0.8,0) {\regionlabel{d-1-i}};
	    \draw[->] (-0.4,0) to (-0.4,1);
	    \draw[->] (0,0) to [out=90, in=270] (0.4,0.5) to   (0.4,1);
	    \draw[->] (0.4,0) to  [out=90, in=270]  (0,0.5) to (0,1);
	    \redcircle{(0,0.75)};
	    \draw[->] (0.8,0) to (0.8,1);
	    \node at (1.2,0.5) {$\cdots$};
	    \node[below] at (1.2,0) {\regionlabel{i-1}};
	    \draw[->] (1.6,0) to (1.6,1);
    \end{tikzpicture}
    \ -\
    \begin{tikzpicture}[anchorbase]
	    \draw[->] (-1.2,0) to (-1.2,1);
	    \node at (-0.8,0.5) {$\cdots$};
	    \node[below] at (-0.8,0) {\regionlabel{d-1-i}};
	    \draw[->] (-0.4,0) to (-0.4,1);
	    \draw[->] (0,0) to (0,0.5) to [out=90, in=270]  (0.4,1);
	    \draw[->] (0.4,0) to (0.4,0.5) to [out=90, in=270]  (0,1);
	    \redcircle{(0.4,0.25)};
	    \draw[->] (0.8,0) to (0.8,1);
	    \node at (1.2,0.5) {$\cdots$};
	    \node[below] at (1.2,0) {\regionlabel{i-1}};
	    \draw[->] (1.6,0) to (1.6,1);
	\end{tikzpicture}
	\ =\
	\sum_{b \in B}
	 \begin{tikzpicture}[anchorbase]
	    \draw[->] (-1.2,0) to (-1.2,1);
	    \node at (-0.8,0.5) {$\cdots$};
	    \node[below] at (-0.8,0) {\regionlabel{d-1-i}};
	    \draw[->] (-0.4,0) to (-0.4,1);
	    \draw[->] (0,0) to (0,1);
	    \bluedot{(0,0.5)} node[anchor=east, color=black] {\dotlabel{b}};
	    \draw[->] (0.4,0) to (0.4,1);
	    \bluedot{(0.4,0.5)} node[anchor=west, color=black] {\dotlabel{\chk{b}}};
	    \draw[->] (1,0) to (1,1);
	    \node at (1.4,0.5) {$\cdots$};
	    \node[below] at (1.4,0) {\regionlabel{i-1}};
	    \draw[->] (1.8,0) to (1.8,1);
	\end{tikzpicture}.
\]
Thus the endomorphism algebra is an affine wreath product algebra.
\end{proof}
It worth mentioning that $\End_{\AW(A)}(\uparrow^{\otimes d})$, as a $\kk$-module, is isomorphic to
\[
    \kk[x_1,\dots,x_d] \otimes A^{\otimes d} \otimes \kk S_d.
\]
This is because each morphism in the basis, by the relations, is a linear combination of morphisms of the form
\[
    t_d^{l_d} \dots t_1^{l_1} u_d(a_d) \dots u_1(a_1) \pi, \quad l_i \in \N,\ a_i \in A,\  \pi \in S_d.
\]
Thus there is a natural surjective linear map determined by
\begin{align}
\begin{split}
    \kk[x_1,\dots,x_d] \otimes A^{\otimes d} \otimes \kk S_d\  \to\ & \End_{\AW(A)}(\uparrow^{\otimes d}), \\
     x_d^{l_d} \dotsm x_1^{l_1} \otimes a_d \otimes \dotsb \otimes a_1 \otimes \pi  \ \mapsto\ & t_d^{l_d} \dotsm t_1^{l_1} u_d(a_d) \dotsm u_1(a_1) \pi.
\end{split}
\end{align}
Hence it suffices to check the injectivity. One can refer to \cite[Theorem 4.6]{Sav17} for a detailed proof.

\color{black}
\subsection{Quantum affine wreath product algebras}

One can also define affine versions of the Hecke category and, more generally, quantum versions of the affine wreath product category.  We refer the reader to \cite{BS18a,BS18b} for further details.

\color {black}


\bibliographystyle{alphaurl}
\bibliography{biblist-2}
\par
\end{document}